\documentclass[12pt]{amsart}
\usepackage{amsmath, amssymb, amsfonts, mathtools, amscd, amsthm, color, latexsym, graphicx, tensor, array, newclude, stmaryrd, mathrsfs}
\usepackage{pdfpages}
\usepackage{graphicx}
\usepackage{dirtytalk}
\usepackage{tikz-cd}

\usepackage{geometry}

\newtheorem{thm}{Theorem}[section]
\newtheorem{prop}[thm]{Proposition}
\newtheorem{lem}[thm]{Lemma}
\newtheorem{rem}[thm]{Remark}

\newtheorem{ex}[thm]{Example}
\newtheorem{defn}[thm]{Definition}

\newcommand{\bb}[1]{\mathbb{#1}}
\newcommand{\fort}[1]{\mathscr{#1}}
\newcommand{\cc}[1]{\mathcal{#1}}

\newcommand{\comp}{\poly_{\cc{F}}(D)}
\newcommand{\rcomp}{\poly_{\cc{F},r}(D)}
\newcommand{\format}{O_{\cc{F}}(1)}
\newcommand{\rformat}{O_{\cc{F},r}(1)}

\newcommand{\C}{\mathcal{C}}
\newcommand{\F}{\mathscr{F}}
\newcommand{\G}{\mathscr{G}}

\DeclareMathOperator{\poly}{poly}
\DeclareMathOperator{\polyl}{poly_{\ell}}

\DeclareMathOperator{\length}{length}
\usepackage{hyperref}
\hypersetup{
    colorlinks=true,
    linkcolor=blue,
    filecolor=blue,      
    urlcolor=blue,
    pdfpagemode=FullScreen
    }
\def\so{\raisebox{.5ex}{\scalebox{0.6}{\#}}\kern-.02em{}o}
\def\s{\raisebox{.8ex}{\scalebox{0.6}{\#}}}

\title{Fortifying the Yomdin-Gromov Algebraic Lemma}
\author{Dmitry Novikov, Benny Zack}
\date{\today}
\begin{document}
\maketitle
\begin{abstract}
We provide sharp cylindrical parametrizations of cylindrical cell decompositions by maps with bounded $C^{r}$ norm in the \so-minimal setting, thus generalizing and strengthening the Yomdin-Gromov Algebraic Lemma. 

We introduce forts, geometrical objects encoding the combinatorial structure of cylindrical cell decompositions in o-minimal geometry. Cylindical decompositions, refinements of such decompositions, and cylindrical parametrizations of such decomposition become morphisms in the category of forts. We formulate and prove the above results in the language of forts. 
\end{abstract}
\tableofcontents
\section{Introduction}
\subsection{The Yomdin Gromov Algebraic Lemma}
Denote $I=(0,1)$. For $m\geq n$ we denote by $\pi^{m}_{n}:I^{m}\to I^{n}$ the projection on the first $n$ coordinates. Often we will omit $m$ from the notation.

Additionally, the symbol $O_{a}(1)$ denotes a specific universally fixed function $a\mapsto C_{a}$ where $C_{a}>0$, and the symbol $\poly_{a}(b)$ denotes a polynomial $p_{a}$ with positive coefficients evaluated at $b$, where $a\mapsto p_{a}$ is a universally fixed map.  
\begin{defn}
Let $U\subset\bb{R}^{\ell}$ be a domain. For a $C^{r}$ function $f:U\to I$, we denote $||f||$ to be its supremum norm, and define
\begin{equation*}
    ||f||_{r}:=\underset{|\alpha|\leq r}{\max}\frac{||f^{(\alpha)}||}{\alpha!}.
\end{equation*}
For a $C^{r}$ map $f:U\to\bb{R}^{m}$, we define $||f||_{r}:=\underset{1\leq i\leq m}{\max}||f_i||_{r}$ where $f_{i}$ are the coordinate functions of $f$. If a $C^{r}$ function (resp. map) satisfies $||f||_{r}\leq 1$ we shall call it an $r$-function (resp. $r$-map).
\end{defn}
The Yomdin-Gromov algebraic lemma and its generalizations to the o-minimal setting have remarkable applications in the fields of dynamics and diophantine geometry. We now state the original semialgebraic version. Fix $r\in\bb{N}$.
\begin{thm}[\cite{Gromov}, Section 3.3]
\label{thm:ayg}
Let $X\subset I^{\ell}$ be a $\mu$-dimensional semialgebraic set, and let $\beta$ the sum of the degrees of the equations and the inequalities that define $X$. Then there exists a constant $C=C(\beta,r,\ell)$ and $r$-maps  $f_1,\dots,f_{C}:I^{\mu}\to X$, such that $X=\cup_{i}f_i(I^{\mu})$.
\end{thm}
\begin{rem}
\label{rem:cellshout}
G.Binyamini and D.Novikov \cite{ComplexCells} prove that $C$ can be bounded from above by $\polyl(\beta)\cdot r^{\mu}$, and moreover that $f_i$ may be taken to be semialgebraic of complexity $\polyl(\beta,r)$. 
\end{rem}
Pila and Wilkie \cite{Pila-Wilkie} generalized Theorem \ref{thm:ayg} to the o-minimal setting in order to prove their celebrated Pila-Wilkie counting theorem about rational points on definable sets. We will now state their version. In the general o-minimal setting the notion of complexity is not available, and will be replaced by uniformity over families. For an introduction on o-minimal sturtures, see \cite{tame}. Fix an o-minimal structure, and an $r\in\bb{N}$.
\begin{thm}[\cite{Pila-Wilkie}, Corollary 5.2]
\label{thm:oyg}
Let $\{X_{\lambda}\subset I^{\ell}\}_{\lambda\in I^{k}}$ be a family of  definable sets such that $\dim X_{\lambda}\leq\mu$ for all $\lambda$. Then there exists a constant $C=C(X,r)$ such that for each $\lambda$ there are definable maps $f_{1,\lambda},\dots,f_{C,\lambda}:I^{\mu}\to X_{\lambda}$ whose images cover $X_{\lambda}$. 
\end{thm}
For a more detailed exposition on this lemma, its applications and limitations, see \cite{ComplexCells,ALR}. In \cite{ALR}, Binyamini and Novikov strengthen Theorem \ref{thm:oyg} by showing that the maps $f_{i,\lambda}$ can be chosen to be \emph{cellular}, see Definition \ref{defn:cellular_map} below. 
\begin{defn}
A basic cell $C\subset\bb{R}^{\ell}$ of length $\ell$ is a product of $\ell$ points $\{0\}$ and intervals $I$.
\end{defn}
\begin{rem}
While a basic cell $C$ is not generally a domain, we will often implicitly identify it with the basic cell obtained by omitting the \{0\} factors from $C$. In particular there is a natural meaning for a map $f:C\to\bb{R}$ to be $C^{r}$.
\end{rem}
\begin{defn}
\label{defn:cellular_map}
Let $X,Y\subset\bb{R}^{\ell}$. A map $f=(f_{1},\dots,f_{\ell}):X\to Y$ is called \emph{precellular} if \begin{enumerate}
    \item It is triangular, that is, the coordinate function $f_i$ depends only on $x_1,\dots,x_{i}$ for every $i=1,\dots,\ell$.
    \item For every $1\leq i\leq\ell$ and every fixed $x_1,\dots,x_{i-1}$, the function $f_i(x_1,\dots,x_{i-1},\cdot)$ is a strictly increasing function.
\end{enumerate}
A cellular map is a continuous precellular map. 
\end{defn}
\begin{rem}
\label{rem:autunicell} If $f:X\to Y$ is cellular, then  for any $1\leq j\leq\ell$ the map $f_{1...i}:=(f_1,\dots,f_i):\pi_{i}(X)\to\pi_{i}(Y)$ is cellular as well.
Also, for any $1\leq i\leq n$, if $(x_1,\dots,x_{i})\in\pi_i(X)$ the map $f(x_1,\dots,x_{i},\cdot):\pi_{i}^{-1}(x_1,\dots,x_{i})\to\pi_i^{-1}(f_{1\dots i}(x_1,\dots,x_{i}))$ is cellular. 
\end{rem}
We now state the main result of \cite{ALR}. The notation $S_{\ell}$ and $F_{\ell}$ below come from \say{set} and \say{function} respectively.
\begin{thm}[\cite{ALR}, Theorem 19]
\label{thm:cyg}
Let $\ell,r\in\bb{N}$, then:\\
$S_{\ell}$: For every definable set $X\subset I^{\ell}$ there exists a finite collection $\{C_{\alpha}\}$ of basic cells of length $\ell$ and a collection of cellular $r$-maps $\{\phi^\alpha:C_{\alpha}\to X\}$ such that $X=\cup_{\alpha}\phi^{\alpha}(C_{\alpha})$. \\ \\
$F_{\ell}$: For every pair $(X,F)$ of a definable set $X\subset I^{\ell}$ and a definable map $F:X\to I^{q}$ (for any $q\in\bb{N})$ there exists a finite collection $\{C_{\alpha}\}$ of basic cells of length $\ell$ and a collection of cellular $r$-maps $\{\phi^\alpha:C_{\alpha}\to X\}$ such that $X=\cup_{\alpha}\phi^{\alpha}(C_{\alpha})$, and for each $\alpha$ the map $(\phi^{\alpha})^{*}F$ is an $r$-map.
\end{thm}
\begin{rem}
This formulation is automatically uniform over families, due to Remark \ref{rem:autunicell}. Also, note that $F_{\ell}$ follows from $S_{\ell+1}$.
\end{rem}
Though the proof of \cite[Theorem 19]{ALR} is less technically involved than previous proofs, it does not provide polynomial bounds in $q$ in the semi-algebraic or Pfaffian cases. Originally one of the main purposes of this paper was to augment the proof of Theorem \ref{thm:cyg} to obtain polynomial bounds in $q$ in these cases. However, due to the recent development of sharply o-minimal structures (\so-minimal for short, see \cite{Wilkie_conj_Proof,SharpIMC,sharp_intro}), it is natural to generalize this proof to work in the setting of \so-minimal structures.

Recall that in \so-minimal structures every definable set is associated with two natural numbers, \say{format} $\cc{F}$ and \say{degree} $D$, generalizing ambient dimension and complexity respectively from semialgebraic geometry. For the complete definition, see Definition \ref{defn:so-minimality}. For a detailed introduction, see \cite{sharp_intro}.

Our main result (see Theorems \ref{thm:mri},\ref{thm:mris} below) in particular implies the following version of Theorem \ref{thm:cyg} for \so-minimal structures, with polynomial bounds in $q$.
\begin{thm}
\label{thm:for_show}
Let $\Sigma=(\mathbf{S},\Omega)$ be a sharply o-minimal structure.\\
\s $S_{\ell}$: Let $X\subset I^{\ell}$ have format $\cc{F}$ and degree $D$. Then there exists a collection $\{C_{\alpha}\}$ of $\comp$ basic cells and cellular $r$-maps $\{\phi^{\alpha}:C_{\alpha}\to X\}$ such that $X=\cup_{\alpha}\phi^{\alpha}(C_{\alpha})$. \\ \\ 
\s $F_{\ell}$: Let a definable set $X\subset I^{\ell}$ have format $\cc{F}$ and degree $D$, and a definable map $F:X\to I^{q}$ such that for every $j=1,\dots,q$ the coordinate functions $F_{j}$ of $F$ have format $\cc{F}$ and degree $D$. Then there exists a collection $\{C_{\alpha}\}$ of $\poly_{\cc{F}}(D,q)$ basic cells of length $\ell$ and a collection of cellular $r$-maps $\{\phi^\alpha:C_{\alpha}\to X\}$ of format $\format$ and degree $\poly_{\cc{F}}(D)$ such that $X=\cup_{\alpha}\phi^{\alpha}(C_{\alpha})$, and for each $\alpha$ the map $(\phi^{\alpha})^{*}F$ is an $r$-map.
\end{thm}
The statement \s $S_{\ell}$ in the (restricted) subPfaffian case is due to Binyamini, Jones, Schmidt and Thomas \cite{when}, and a rewording of their proof works in the general \so-minimal case. The statement \s $F_{\ell}$ however is stronger and it does not follow from \s $S_{\ell+1}$. In fact, our main result is notably stronger than Theorem \ref{thm:for_show}. We prove the existence of a cellular $r$-parametrization with strict control on the combinatorial and geometric structure of the parametrizing maps. See Theorems \ref{thm:mri},\ref{thm:mris} below. We start with reviewing the notion of cylindrical decomposition, and introducing the notion of cylindrical parametrization. 

\subsection{Cells and Cylindrical Decompositions}
Fix an o-minimal expansion of $\bb{R}$. We review the notions of cells as they are presented in \cite{tame}.
\begin{defn}[Cells]
\label{defn:cells}
A cell of length $1$ is a subset $\C\subset\bb{R}^1$ which is either a point or an open interval, and its type is defined to be either $(0)$ or $(1)$ respectively. \\ \\
Let $\C\subset\bb{R}^{\ell}$ be a cell of length $\ell$ with type $\tau\in\{0,1\}^{\ell}$, and let $f,g:\C\to\bb{R}_{>0}$ be continuous definable functions with $f<g$ everywhere. A cell $\widetilde{\C}$ of length $\ell+1$ is one of the following sets:
\begin{itemize}
    \item $\widetilde{\C}=\C\odot(f,g):=\{(x,y):\;x\in\C,\;f(x)<y<g(x)\}$, in which case the type of $\widetilde{\C}$ is $(\tau,1)\in\{0,1\}^{\ell+1}$.
    \item $\widetilde{\C}=\C\odot f:=\{(x,y):\;x\in\C\,;y=f(x)\}$, in which case the type of $\widetilde{\C}$ is $(\tau,0)\in\{0,1\}^{\ell+1}$.
\end{itemize}
A cell $\C$ is \emph{compatible} with a set $X$ if either $\C\subset X$ or $\C\cap X=\emptyset$. A collection $\{\C_{1},\dots,\C_{n}\}$ is a cellular decomposition of $X$ if the cells $\C_{i}$ are pairwise disjoint and $\cup_{i}\C_{i}=X$. 
\end{defn}
We will use the following stronger notion, see Figure 1.
\begin{defn}[Cylindrical decompositions]\label{defn:cyldec}
Let $X\subset I$ be a definable set. A cylindrical decomposition of $X$ is a cellular decomposition of $X$. Let $X\subset I^{\ell}$ be definable. A cellular decomposition $\Phi$ of $X$ is called a cylindrical decomposition of $X$ if the collection $\pi_{\ell-1}(\Phi):=\{\pi_{\ell-1}(\cc{C})|\;\cc{C}\in\Phi\}$ is a cylindrical decomposition of $\pi_{\ell-1}(X)$.
\end{defn} 
\begin{figure}
\centering
	\includegraphics[width=0.8\linewidth]{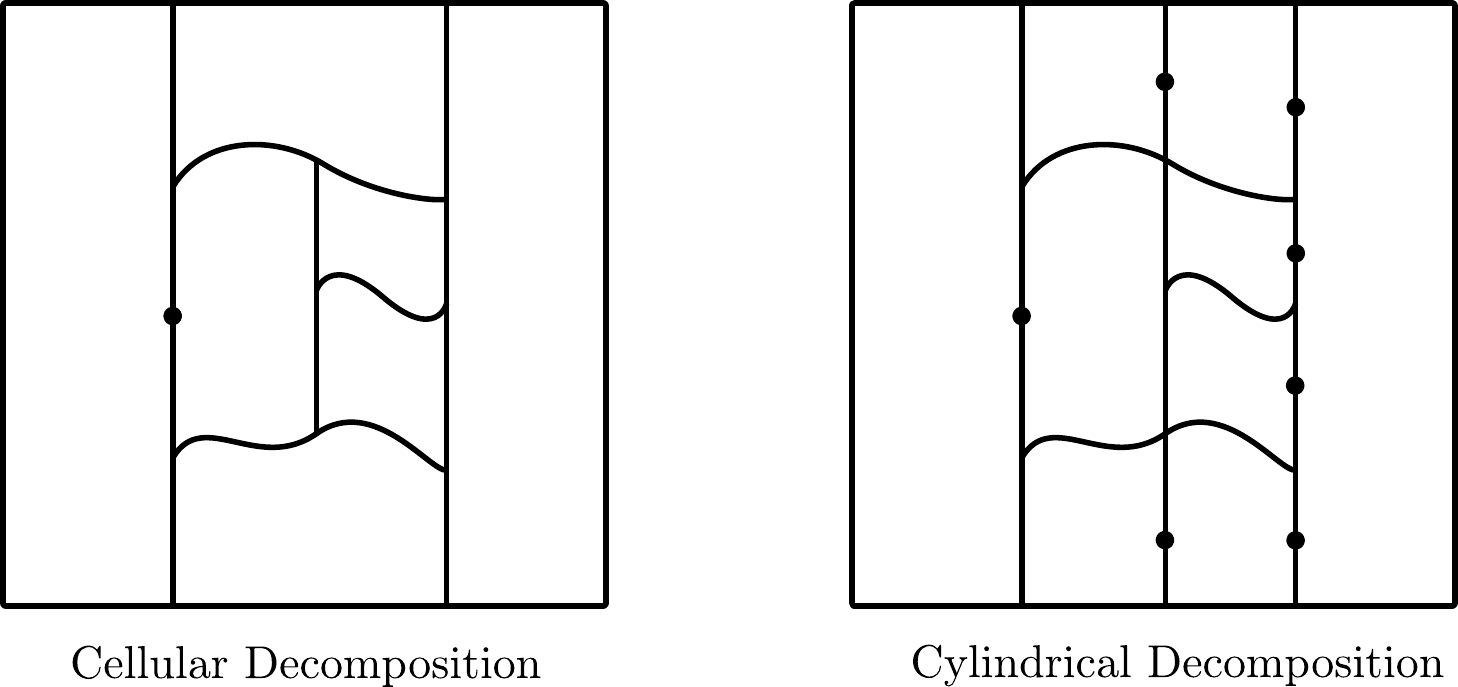}
	\caption{An example of a cellular decomposition of $I^{2}$ which is not a cylindrical decomposition, and a cylindrical decomposition of $I^{2}$ \emph{refining} (see Definition \ref{defn:cyl_dec_refine}) it.}
	\label{fig:cell_vs_cyl}
\end{figure}
\subsection{Sharp cylindrical decomposition}
Let us recall the definition of sharp cellular decomposition from \cite{sharp_intro}.
\begin{defn}
Let $\mathbf{S}$ be an o-minimal structure and $\Omega$ be an FD-filtration on $\mathbf{S}$. We say that $(\mathbf{S},\Omega)$ has \emph{sharp cylindrical decomposition} (or \emph{\s CD} for short) if whenever $X_1,\dots, X_{s}\subset I^{\ell}$ are definable sets of format $\cc{F}$ and degree $D$, there exists a cylindrical decomposition of $I^{\ell}$ into $\poly_{\cc{F}}(D,s)$ cells of format $O_{\cc{F}}(1)$ and degree $\poly_{\cc{F}}(D)$ compatible with $X_1,\dots,X_{s}$.
\end{defn}
It is not generally known if every \so-minimal structure has \s CD. However, it was shown in \cite{sharp_intro} that for quantitative applications one can always reduce to the case of a \so-minimal structure with \s CD, see \cite[Remark 1.3, Theorem 1.9]{sharp_intro} for more details. We will therefore assume that our \so-minimal structure has \s CD, and this will be sufficient for quantitative applications.

For example, while our main results are not known to be true as stated for general \so-minimal structures, they are sufficient to prove a sharp version of the Pila-Wilkie counting theorem for every \so-minimal structure, see section \ref{sec:applications} for more details.
\subsection{Cylindrical Parametrizations and the Main Result}
Our goal is to strengthen Theorem \ref{thm:cyg} in such a way that the maps $\phi^{\alpha}$ have additional combinatorial properties. More precisely, we prove that they can be chosen to form a \emph{Cylindrical parametrization}, see Definition \ref{defn:cylparam} below.  
\begin{defn}
\label{defn:cylparam}
Let $\Phi=\{\C_{\alpha}\}_{\alpha\in A}$ be a cylindrical decomposition of $X\subset I^{\ell}$. A cylindrical parametrization of $\Phi$ is a collection of surjective cellular maps $\{\phi^{\alpha}:C_{\alpha}\to\C_{\alpha}\}_{\alpha\in A}$, where $C_{\alpha}$ is a basic cell of the same type as $\C_{\alpha}$, such that the following holds: \\ 
For any two cells $\C_{\alpha},\C_{\gamma}$, and for every $1\leq k\leq\ell-1$, if $\pi_{k}(\C_{\alpha})=\pi_{k}(\C_{\gamma})$ then $\phi^{\alpha}_{k}\circ\pi_{k}=\phi^{\gamma}_{k}\circ\pi_{k}$. 
\end{defn}
\begin{ex}
A cylindrical parametrization of the cylindrical decomposition from Figure 1 can be viewed as a surjective, precellular map $\phi:\mathscr{F}\to I^{2}$, where $\mathscr{F}$ is the set in  figure 2 below. Notice that $\mathscr{F}$ has a natural decomposition into cells that are integer translations of basic cells, and $\phi$ is continuous when restricted to any such cell. The sets $\mathscr{F},I^{2}$ are examples of \emph{forts}, and $\phi$ is an example of a \emph{morphism} between forts, see Definition \ref{defn:fortmor}.
\end{ex}
\begin{figure}
\centering
	\includegraphics[height=0.5\linewidth]{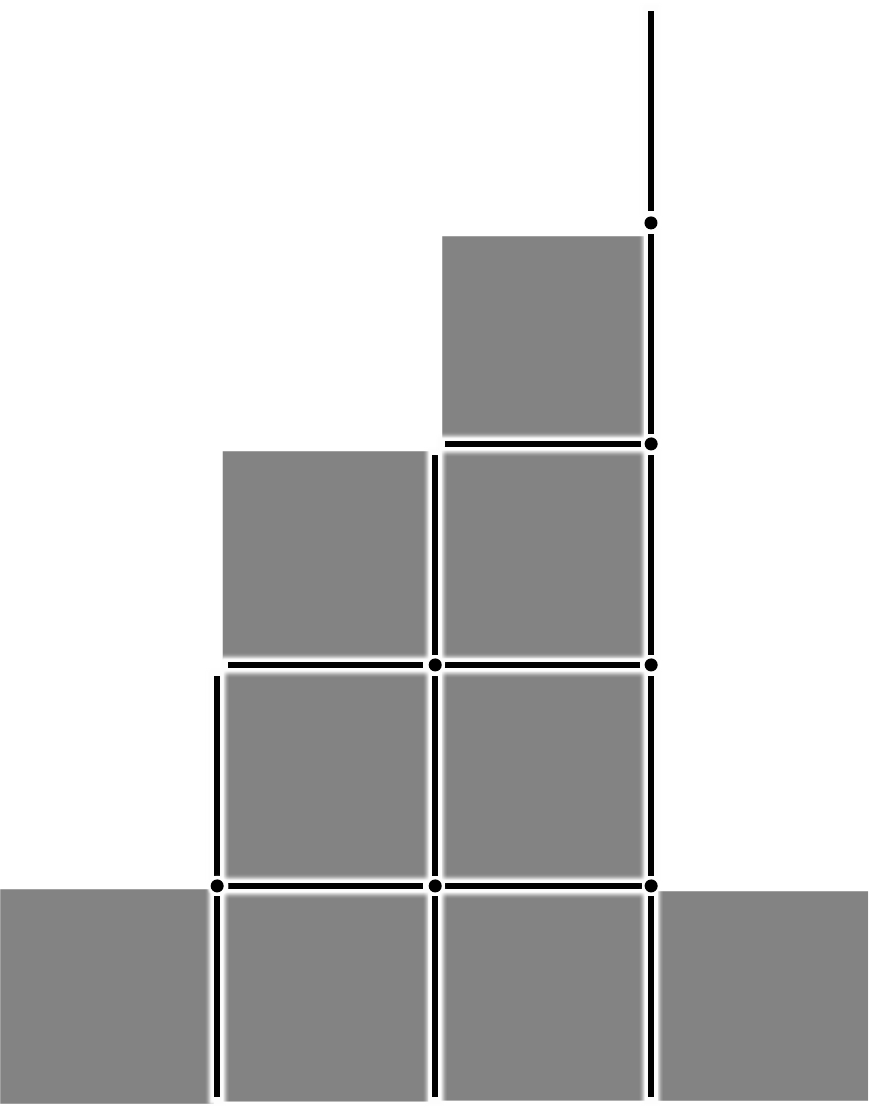}
	\caption{The fort
	     $(0,1)\times (0,1)\bigcup\{1\}\times(0,2)\bigcup(1,2)\times(0,3)\bigcup\{2\}\times(0,3)\bigcup(2,3)\times(0,4)\bigcup\{3\}\times(0,5)\bigcup(3,4)\times(0,1)$.}
	\label{fig:fort}
\end{figure}
\begin{defn}
\label{defn:cyl_dec_refine}
Given two cylindrical decompositions $\Phi=\{\C_{\alpha}\}_{\alpha\in A}$ and $\Phi'=\{\C_{\beta}\}_{\beta\in B}$ of a set $X$, we say that $\Phi'$ is a refinement of $\Phi$ if there exists a map $s:B\to A$ such that $\C_{\beta}\subset\C_{s(\beta)}$ for all $\beta\in B$.
\end{defn}
A cylindrical parametrization will be called an $r$-parametrization if all its maps are $r$-maps. We now state our main results.
\begin{thm}\label{thm:mri} Fix $\ell,r\in\bb{N}$. \\
$S^{*}_{\ell}$: Let $\Phi$ be a cylindrical decomposition of $I^{\ell}$. Then there exists a refinement $\Phi'$ of $\Phi$ that admits a cylindrical $r$-parametrization. \\ \\
$F^{*}_{\ell}$: Let $\Phi=\{\C_{\alpha}\}_{\alpha\in A}$ be a cylindrical decomposition of $I^{\ell}$, together with a collection of definable maps $\{f_{\C_{\alpha},j}:\C_{\alpha}\to I\}_{\alpha\in A\;,j\in J}$.Then there exists a refinement $\Phi'=\{\C'_{\gamma}\}_{\gamma\in A'}$ of $\Phi$ that admits a cylindrical $r$-parametrization $\{\phi^{\gamma}:C_{\gamma}\to\C'_{\gamma}\}_{\gamma\in A'}$ such that if $\C'_{\gamma}\subset\C_{\alpha}$, then $(\phi^{\gamma})^{*}f_{\C_{\alpha},j}$ is an $r$-function for every $j$.
\end{thm}
This theorem can be rather easily deduced from Theorem \ref{thm:cyg}. However, its sharp counterpart is much more meaningful.
\begin{thm}\label{thm:mris} Fix $\ell,r\in\bb{N}$. \\
$S^{*}_{\ell}$: Let $\Phi$ be a cylindrical decomposition of $I^{\ell}$ of size $N$ into cells of format $\cc{F}$ and degree $D$. Then there exists a refinement $\Phi'$ of $\Phi$ of size $\poly_{\cc{F},r}(D,N)$, such that $\Phi'$ admits a cylindrical $r$-parametrization whose maps have format $O_{\cc{F},r}(1)$ and degree $\poly_{\cc{F},r}(D)$. \\ \\
$F^{*}_{\ell}$: Let $\Phi=\{\C_{\alpha}\}_{\alpha\in A}$ be a cylindrical decomposition of $I^{\ell}$ of size $N$ into cells of format $\cc{F}$ and degree $D$, together with a collection of definable maps $\{f_{\C_{\alpha},j}:\C_{\alpha}\to I\}_{\alpha \in A,\;j\in J}$ such that each $f_{\C_{\alpha,j}}$ has format $\cc{F}$ and degree $D$. Then there exists a refinement $\Phi'=\{\C'_{\gamma}\}_{\gamma\in A'}$ of $\Phi$ of size $\poly_{\cc{F},r}(D,N,|J|)$, and a cylindrical $r$-parametrization $\{\phi^{\gamma}:C_{\gamma}\to\C'_{\gamma}\}$ of $\Phi'$, such that the maps $\phi^{\gamma}$ have format $\rformat$ and degree $\poly_{\cc{F},r}(D)$, where $(\phi^{\gamma})^{*}f_{\C_{\alpha},j}$ is an $r$-function for every $j$ whenever  $\C'_{\gamma}\subset\C_{\alpha}$.
\end{thm}
\begin{rem}(On the index set $J$)
A priopri, $J$ may depend on $\alpha$. However we can, and always will, add constant functions to the collections so that the index set $J$ can be chosen to be independent of the cells $\C$. 
\end{rem}
Our main results generalize Theorem \ref{thm:cyg} (and Theorem \ref{thm:for_show}) in the following way. Let $X_1,\dots,X_{s}\subset I^{\ell}$ be definable. By o-minimality, there exists a cylindrical decomposition $\Phi$ compatible with $X_i$ for $1\leq i\leq s$. By Theorem \ref{thm:mri}, we may assume that $\Phi$ admits a cylindrical $r$-parametrization $\{\phi^{\alpha}\}_{\alpha\in I}$. Fix $i$, and let $\{\C_\beta\}_{\beta\in J}$ be the cells of $\Phi$ contained in $X_i$. Then the collection $\{\phi^{\beta}\}_{\beta\in J}$ (which is a cylindrical $r$-parametrization of $X_i$) satisfies the requirements of Theorem \ref{thm:cyg} for $X_i$. The proof of Theorem \ref{thm:for_show} is analogous. Indeed, if the sets $X_i$ are defined in a sharply o-minimal structure with \s CD, we obtain sharp cylindrical parametrization using \s CD and Theorem \ref{thm:mris} instead of ordinary cylindrical decomposition and Theorem \ref{thm:mri}. 

\begin{rem}
Binyamini and Novikov have shown that in Theorem \ref{thm:ayg}, the complexity of the parametrizing maps can be taken to be polynomial in $r$. This is essentially due to the analytic nature of the semialgebraic structure. In the general sharp case, we are unable obtain polynomial bounds in $r$. However, in \cite{Wilkie_conj_Proof}, Binyamini, Novikov and B.Zack prove an approximation version of the Yomdin-Gromov Lemma with polynomial dependence in $r$ under the assumption of sharp derivatives (see \cite{Wilkie_conj_Proof} or section 3 in \cite{sharp_intro}). While we conjecture that with sharp derivatives the dependence on $r$ in Theorems \ref{thm:mri},\ref{thm:mris} can be taken to by polynomial, we do not see a clear way to utilize sharp derivatives in this direction. 
\end{rem}
\begin{rem}[Effectivity]
If one begins with an effective \so-minimal structure in the sense of \cite[Remark 1.24]{sharp_intro}, then all of our results are effective in the same sense.
\end{rem}
\subsection{Applications of the Main Result}
\label{sec:applications}
Binyamini, Jones, Schmidt and Thomas \cite{when} prove a sharp version of the Pila-Wilkie counting theorem (see \cite{Pila-Wilkie}) for the restricted subPfaffian structure using the results of \cite{Pf} and a version of \s $S_{\ell}$ from Theorem \ref{thm:for_show} for this structure. The same argument will yield a sharp version of the Pila-Wilkie counting theorem for any \so-minimal structure with sharp cylindrical decomposition, and thus for any \so-minimal structure, see Theorem \ref{thm:sharp_pila_wilkie} below. We note that under the assumption of \emph{sharp derivatives}, a $\textnormal{polylog}$ version of the Pila-Wilkie Theorem generalizing the Wilkie conjecture, holds, see \cite[Definition 8, Theorem 1]{Wilkie_conj_Proof}.
\begin{defn}
Let $A\subset\bb{R}^{\ell}$. We denote $A^{\textnormal{alg}}$ to be the union of all connected $1$ dimensional semialgebraic subsets of $A$, and $A^{\textnormal{tran}}=A\backslash A^{\textnormal{alg}}$.
\end{defn}
\begin{defn}
For a rational number $x=\frac{p}{q}$ where $p,q$ are coprime integers, we denote $H(x):=\max\{|p|,|q|\}$, $H(x)$ is called the height of $x$. For a vector $x\in\bb{Q}^{\ell}$ we denote $H(x)$ to be the maximum among the height of $x'$s coordinates.
\end{defn}
\begin{thm}
\label{thm:sharp_pila_wilkie}
Let $(\cc{S},\Omega)$ be a sharply o-minimal structure. Let $A\subset\bb{R}^{\ell}$ be a definable set of format $\cc{F}$ and degree $D$. Then for every $\epsilon>0$ we have \begin{equation}
    |\{x\in A^{\textnormal{tran}}\cap\bb{Q}^{\ell}:\;H(x)\leq H\}|\leq\poly_{\cc{F},\epsilon}(D)H^{\epsilon}.
\end{equation}
\end{thm}
\subsection{Structure of this paper}
 In \cite{ALR}, standard o-minimal technique is used to conclude a family version for $F_{\ell}$ from $F_{\ell}$. This technique is not sharp. For example, the dimension of the space of all semialgebraic sets with complexity $\leq \beta$ is polynomial in $\beta$, and so we cannot \say{interact} with it in the framework of \so-minimal structures. Instead, we prove a family version for every single step, keeping track of the formats and degrees of the total spaces involved.\\ \\ In section 2, we review sharp versions of some standard o-minimal constructions. 
 \\ \\  In section 3, we formally introduce forts and study their elementary properties and constructions. We reformulate $S^{*}_{\ell},F^{*}_{\ell}$ in terms of forts, and denote them by $S_{\ell,k},F_{\ell,k}$, where $k$ is the amount of parameters in the family. \\ \\
 We prove $S_{\ell,k},F_{\ell,k}$ by induction. We first prove $F_{1,k}$ for every $k\in\bb{Z}_{\geq 0}$. We then show the steps $S_{\leq\ell,k}+F_{\leq\ell,k}\to S_{\ell+1,k}$ and $S_{\ell+1,k}+F_{\leq\ell,k+1}\to F_{\ell+1,k}$. Sections 4,5,6 are dedicated to $F_{1,k}$ and the two induction steps respectively. \\ \\ 
 Finally, we remark that our proof for Theorem \ref{thm:mris} works completely verbatim for Theorem \ref{thm:mri}, one just needs to ignore the  FD-filtration and carry through the same constructions. Therefore we will always work with sharply o-minimal structures. 
 \section{Preliminary results for Sharp Structures}
 \subsection{Definition of sharply o-minimal structures}
 We follow the definition for \so-minimal structures as it appears in \cite{sharp_intro}. 
 \begin{defn}[FD-filtrations]
 \label{defn:FD-filtration}
 Let $\mathbf{S}$ be an o-minimal expansion of the real field. An FD-filtration on $\mathbf{S}$ is a filtration $\Omega=\{\Omega_{\cc{F},D}\}_{\cc{F},D\in\bb{N}}$ on the collection of definable sets by two natural numbers,  such that the following holds.
 \begin{enumerate}
     \item Every definable set is in $\Omega_{\cc{F},D}$ for some $\cc{F},D\in\bb{N}$,
     \item For every $\cc{F},D\in\bb{N}$ we have \begin{equation}
         \Omega_{\cc{F},D}\subset \Omega_{\cc{F}+1,D}\cap \Omega_{\cc{F},D+1}.
     \end{equation}
 \end{enumerate}
 We say that a definable set has format $\cc{F}$ and degree $D$ if $X\in\Omega_{\cc{F},D}$. We say that a definable map has format $\cc{F}$ and degree $D$ if its graph has format $\cc{F}$ and degree $D$.
 \end{defn}
 \begin{defn}[Sharp o-minimality]
 \label{defn:so-minimality}
 Let $\mathbf{S}$ be an o-minimal expansion of the real field, and let $\Omega$ be an FD-filtration on $\mathbf{S}$. The pair $(\mathbf{S},\Sigma)$ is called a sharply o-minimal structure (\so-minimal for short) if the following holds. 
 \begin{enumerate}
     \item For every $\cc{F}\in\bb{N}$ there exists a polynomial $P_{\cc{F}}$ of one variable with positive coefficients such that if $X\subset\bb{R}$ has format $\cc{F}$ and degree $D$ then it has at most $P_{\cc{F}}(D)$ components.
     \item If $X\subset\bb{R}^{\ell}$ has format $\cc{F}$ and degree $D$ then the sets $\bb{R}\times X, X\times\bb{R}, \bb{R}^{\ell}\setminus X$ and $\pi_{\ell-1}(X)$ have format $\cc{F}+1$ and degree $D$. Moreover, $\cc{F}\geq\ell$ neccasarily holds.
     \item If $X_{i}\subset\bb{R}^{\ell}$ are definable sets of format $\cc{F}_{i}$ and degree $D_{i}$ for $i=1,\dots, k$ then the union $\cup X_{i}$ has format $\cc{F}$ and degree $D$, and the intersection $\cap X_{i}$ has format $\cc{F}+1$ and degree $D$, where $\cc{F}:=\underset{i}{\max}\cc{F}_{i}$ and $D:=\sum_{i}D_{i}$.
     \item If $P\in\bb{R}[x_{1},\dots,x_{\ell}]$ has degree $d$, then the zero set $\{P=0\}\subset\bb{R}^{\ell}$ has format $\ell$ and degree $d$. 
 \end{enumerate}
 \end{defn}
 Fix a sharply o-minimal structure. We will often work with families, so we introduce the following definition to make notation easier.
 \begin{defn}
 \label{defn:format_degree_of_family}
Let $F_{\lambda}=\{f_{\lambda}:X\to Y\}_{\lambda\in\Lambda}$ be a family of definable maps. We say that the family $F_{\lambda}$ is an $(\cc{F},D)$-family if the total space $F_{\Lambda}=\{(\lambda,x,y)\in\Lambda\times X\times Y:\;y=f_{\lambda}(x)\}$ has format $\cc{F}$ and degree $D$. Similarly, if $X_{\lambda}$ is a family of subsets of $\bb{R}^{\ell}$, we say it is an $(\cc{F},D)$-family if the total space $X_{\Lambda}=\{(\lambda,x):\;\lambda\in\Lambda,\;x\in X_{\lambda}\}$ has format $\cc{F}$ and degree $D$.
\end{defn}
\subsection{Sharpness of arithmetic operations}
The following basic lemmas about sharpness of arithmetic operations were not strictly treated in \cite{sharp_intro,Wilkie_conj_Proof}, even though they were used there implicitly. We prove them here, for they serve as an additional illustration of the mechanism of the \so-minimal axioms. 
\begin{lem}
If $X_{\lambda}\subset\bb{R}^{a},\;Y_{\lambda}\subset\bb{R}^{b}$ are definable $(\cc{F},D)$-families then the fiber product $X_{\lambda}\times Y_{\lambda}$ is an $(\format,\comp)$ family. 
\end{lem}
\begin{proof}
The non-family version immediately from the axioms and from the equality $X\times Y=\left(X\times\bb{R}^{b}\right)\cap\left(\bb{R}^{a}\times Y\right)$. Therefore the set \begin{equation}
    X_{\Lambda}\times Y_{\Lambda}=\{(\lambda,x,\mu,y):\lambda,\mu\in\Lambda,\;x\in X_{\lambda},\;y\in Y_{\mu}\}
\end{equation}
has format $\format$ and degree $\comp$. It is left to notice that the total space of the family $X_{\lambda}\times Y_{\lambda}$ is given by a projection of $(X_{\Lambda}\times Y_{\Lambda})\cap\{\lambda=\mu\}$. 
\end{proof}
Similarly, one can prove the following Lemma.
\begin{lem}
\label{lem:sharp_arithmetic}
 Let $f_{\lambda},g_{\lambda}:X\to Y$ and $h_{\lambda}:Y\to Z$ be definable $(\cc{F},D)$-families. Then the families $f_{\lambda}\pm g_{\lambda},h_{\lambda}\circ f_{\lambda}$ are $(\format,\comp)$-families. If $Y\subset\bb{R}$, then so is $f_{\lambda}g_{\lambda}$, and if in addition $g_{\lambda}\neq 0$, then so is $f_{\lambda}/g_{\lambda}$. 
\end{lem}
\subsection{Sharpness of Refinement}
 We will often use the following Lemmas implicitly. 
 \begin{lem}
 \label{lem:makecyl}
 \phantom{o}
 \begin{enumerate}
     \item  Let $\Phi$ be a cylindrical decomposition of $I^{\ell}$, whose cells are of format $\cc{F}$ and degree $D$, and for every cell $\C\in\Phi$ let $\Phi_{\C}$ be a cylindrical decomposition of $\C$ whose cells are of format $\cc{F}$ and degree $D$. Then there exists a cylindrical decomposition $\Psi$ of $I^{\ell}$ of size $\poly_{\cc{F}}(D,\sum_{\C\in\Phi}|\Phi_{\C}|)$ whose cells have format $O_{\cc{F}}(1)$ and degree $\comp$ that refines all the decompositions $\Phi_{\C}$. More precisely, every cell of $\Psi$ is contained in a cell of $\Phi_{\C}$ for some $\C\in\Phi$.
     \item Let $\Phi_{1},\dots,\Phi_{s}$ be cylindrical decompositions of $I^{\ell}$ whose cells have format $\cc{F}$ and degree $D$. Then there exists a cylindrical decomposition $\Psi$ of $I^{\ell}$ of size $\poly_{\cc{F}}(D,\sum_{i=1}^{s}|\Phi_{i}|)$ whose cells have format $\format$ and degree $\comp$ that refines $\Phi_1,\dots,\Phi_s$.
 \end{enumerate}
 \end{lem}
 \begin{proof}
 For the first item, use \s CD on all the collection of all cells in $\cup_{\C}\Phi_{\C}$. For the second item, use \s CD on the collection of all cells in $\cup_{i}\Phi_{i}$. 
 \end{proof}
\subsection{Sharpness of \texorpdfstring{$C^{r}$}{Lg} locus}
We need the following slightly stronger formulation of \cite[Proposition 3.1]{sharp_intro}. The proof remains the same however.
\begin{prop}
\label{prop:Cr_locus_effective}
Let $f_{1},\dots,f_{s}:I^{\ell}\to I$ be definable functions of format $\cc{F}$ and degree $D$. Then $f_{1},\dots,f_{s}$ are $C^{r}$ outside a definable set $V$ of codimension $\geq 1$ of format $O_{\cc{F},r}(1)$ and degree $\poly_{\cc{F},r}(D,s)$.
\end{prop}
 \begin{proof}
See the proof of \cite[Proposition 3.1]{sharp_intro}.
\end{proof}
 \begin{rem}
 \label{rem:format_degree_of_derivatives}
 The same proof shows that if $f$ is $C^{r}$, then for any $\alpha\in\bb{N}^{\ell}$ with $|\alpha|\leq r$ the partial derivative $\frac{\partial}{\partial x_{\alpha}}f$ has format $O_{\cc{F},|\alpha|}(1)$ and degree $\poly_{\cc{F},|\alpha|}(D)$. We will use this fact without referring to this remark.  
 \end{rem}
\subsection{Sharp definable choice}
 \begin{prop}
 \label{prop:Effective_definable_choice}
 Let $\{X_{y}\subset\bb{R}^{\ell}\}_{y\in Y}$ be an $(\cc{F},D)$-family of non empty definable sets. Then there exists a definable map $g:Y\to\bb{R}^{\ell}$ of format $\format$ and degree $\comp$ such that $g(y)\in X_{y}$ for all $y\in Y$.
\end{prop}
\begin{proof}
See \cite[Section 4]{sharp_intro}.
\end{proof}
\begin{rem}
The following is an equivalent formulation: Let $f:X\to Y$ be a surjective definable function of format $\cc{F}$ and degree $D$. Then there exists a definable section $g:Y\to X$ of $f$ that has format $\format$ and degree $\comp$.
\end{rem}
\section{Forts}
\subsection{Forts and Morphisms}
 Let $q^{m}_{i}:I^{m}\to I^{m-i}$ be the projection on the last $m-i$ coordinates. $m$ will often be omitted.
 
 For a set $X\subset\bb{R}^{\ell}$ and $x\in\pi_{i}(X)$ where $1\leq i\leq\ell$, we denote $X_{x}:=q_{i}(\pi_{i}|_{X}^{-1}(x))$. For a point $x\in\bb{R}^{\ell}$ and an integer $1\leq i\leq \ell$, we denote $x_{1\dots i}:=(x_{1},\dots,x_{i})$.
 \begin{defn}
An integer cell of length $1$ is either a point $\{k\}$ or an interval $(k,k+1)$ where $k\in\bb{Z}$. An integer cell of length $\ell$ is a product of $\ell$ integer cells of length $1$.
\end{defn}
\begin{rem}
From now on the notation $\C$ is reserved for integer cells, and not general cells as in Definition \ref{defn:cells}.
\end{rem}
 There are two useful inductive definitions of forts. Definition \ref{defn:fortsx} defines forts via their structure along the $x_{1}$ axis, while Definition \ref{defn:fortsi} is similar to the inductive definition of cells in o-minimal geometry. We prove in Proposition \ref{prop:fortsi_vs_fortsx} below that these two definitions are equivalent. 
 \begin{defn}
\label{defn:fortsx}
A fort $\F$ of length $1$ is an interval $(0,n)$ for a positive integer $n$. It is naturally partitioned into integer cells, and we define $C(\F)$ to be the set of these integer cells. We denote by $\textnormal{Forts}(1)$ the set of forts of length $1$. \\ \\ 
Let $\F^{1}$ be a fort of length $1$, and let $s:C(\F^{1})\to \textnormal{Forts}(\ell-1)$ be any function. A fort of length $\ell$ is the set
\begin{equation}
    \F^{1}\ltimes s:=\bigcup_{\C\in C(\F^{1})}\C\times s(\C).
\end{equation}
The fort $\F^{1}\ltimes s$ is naturally partitioned into integer cells $\C\times\C'$ where $\C\in C(\F^{1})$ and $\C'\in C(s(\C))$. We denote by $C(\F^{1}\ltimes s)$ the set of these cells. Finally, denote by $\textnormal{Forts}(\ell)$ the set of forts of length $\ell$. 
\end{defn}
\begin{rem}
We will often write $\F$ with an upper index to indicate its length. i.e. if not stated otherwise $\F^{\ell}$ means that $\F$ is a fort of length $\ell$.
\end{rem}
\begin{defn}
\label{defn:fortsi}
Let $\F^{\ell-1}$ be a fort, and $\varphi:\F^{\ell-1}\to\bb{N}$ a function constant on each cell $\C\in C(\F^{\ell-1})$. Then the set \begin{equation}
    \F^{\ell-1}\odot\varphi:=\{(x_1,\dots,x_{\ell}):\;(x_1,\dots,x_{\ell-1})\in\F^{\ell-1},\;0<x_{\ell}<\varphi(x_1,\dots,x_{\ell-1})\}
\end{equation}
is a fort of length $\ell$. 
\end{defn}
\begin{ex}
\label{ex:fort_proj}
Let $\F^{\ell}$ be a fort. Then for any $1\leq i<\ell$, the set $\pi_{i}(\F^{\ell})$ is a fort. We will also say that $\F^{\ell}$ extends $\pi_{i}(\F^{\ell}).$
\end{ex}
\begin{ex} Let $\F^{\ell}$ be a fort. Then for any $1\leq i\leq\ell$ and every $(x_1,\dots,x_i)\in\pi_{i}(\F^{\ell})$, the set $\F^{\ell}_{(x_1,\dots,x_i)}$ is a fort.
\end{ex}
\begin{prop}
\label{prop:fortsi_vs_fortsx}
Definition \ref{defn:fortsi} is equivalent to Definition \ref{defn:fortsx}.
\end{prop}
\begin{proof}
First note that it is straightforward to verify Example \ref{ex:fort_proj} with both definitions. We argue by induction on $\ell$, the cases $\ell=1,2$ being clear. Let $\F^{\ell}$ be a fort in the sense of Definition \ref{defn:fortsx}, and write $\F^{\ell}=\F^{1}\ltimes s$. For every cell $\C\in C(\F^{1})$, denote by $s'(\C):=\pi_{\ell-2}(s(\C))$. By induction, there exists a function $\varphi_{\C}:s'(\C)\to\bb{N}$, constant on the cells of $s'(\C)$, such that $s'(\C)\odot\varphi_{\C}=s(\C)$. We extend $\varphi_{\C}$ to $\C\times s'(\C)$ by making it not depend on $x_{1}\in\C$, and keep the same notation $\varphi_{\C}$.

Next, note that $\F^{\ell-1}:=\pi_{\ell-1}(\F^{\ell})=\F^{1}\ltimes s'$, so $\varphi:=\bigcup_{\C\in C(\F^{1})}\varphi_{\C}$ is a function $\F^{\ell-1}\to\bb{N}$ which is constant on the cells of $\F^{\ell-1}$. It remains to show that $\F^{\ell}=\F^{\ell-1}\odot\varphi$.

Fix a cell in $\F^{\ell}$, it is of the form $\C'=\C\times\C''$ where $\C\in C(\F^{1})$ and $\C''\in C(s(\C))$. Since $s'(\C)\odot\varphi_{\C}=s(\C)$, the value of $\varphi_{\C}$ on $\pi_{\ell-2}(\C'')$ is not smaller than $\sup q_{\ell-2}(\C'')$. By the definition of $\varphi$, this means that this cell is also a cell of $\F^{\ell-1}\odot\varphi$. The same argument shows that every cell of $\F^{\ell-1}\odot\varphi$ is a cell of $\F^{\ell}$.

The proof that Definition \ref{defn:fortsi} implies Definition \ref{defn:fortsx} is completely analogous. 
\end{proof}
We next equip $\textnormal{Forts}(\ell)$ with morphisms, making it a category. 
\begin{defn}
\label{defn:fortmor}
Let $\F_{1},\F_{2}\in \textnormal{Forts}(\ell)$. A definable map $\phi:\F_{1}\to\F_{2}$ is called a morphism if:
\begin{enumerate}
    \item $\phi$ is surjective
    \item $\phi$ is precellular. 
    \item For every $\C_{1}\in\C(\F_{1})$ there exists $\C_2\in C(\F_{2})$ such that $\phi(\C_1)\subset\phi(C_2)$, and moreover the restriction $\Phi|_{\C_1}$ is continuous.
\end{enumerate}
\end{defn}
\begin{ex}
Let $\F$ be a fort. Then the set $d\F=\{d\cdot x,\;x\in\F\}$ for $d\in\bb{N}$ is also a fort, and the map $d\F\to\F$ given by $x\mapsto d^{-1}x$ is a morphism called linear subdivision of order $d$.
\end{ex}
\begin{ex}
\label{ex:base_morphism}
Let $\phi:\F^{\ell}_{1}\to\F^{\ell}_{2}$ be a morphism. Then for any $1\leq i<\ell$, $\phi_{1\dots i}:=(\phi_1,\dots,\phi_{i})$ is a morphism $\pi_{i}(\F^{\ell}_{1})\to\pi_{i}(\F^{\ell}_{2})$.  
\end{ex}
\begin{ex}
\label{ex:fibre_morphism}
Let $\phi:\F^{\ell}\to\mathscr{G}^{\ell}$ be a morphism. Then for any $1\leq i\leq\ell$, and every $(x_1,\dots,x_i)\in\pi_i(\F^{\ell})$, the map $\phi_{(x_1,\dots,x_{i})}:=\phi(x_1,\dots,x_i,\cdot,\dots,\cdot)$ is a morphism $\F^{\ell}_{(x_1,\dots,x_i)}\to\mathscr{G}^{\ell}_{\phi_{1\dots i}(x_1,\dots,x_i)}$.
\end{ex}
The following is an important example of morphisms. They are constructed using the canonical coordinate-wise affine parmaterization of o-minimal cells.
\begin{defn}
Let $X\subset\bb{R}^{m},Y\subset\bb{R}^{n}$ be definable sets. A map $f:X\to Y$ is called \emph{affine} if $f$ is a restriction of an affine map $\bb{R}^{m}\to\bb{R}^{n}$, i.e if $f=Ax+b$ for $A\in M_{n\times m},b\in\bb{R}^{n}$.
\end{defn}
\begin{defn}
Let $\C$ be an integer cell of length $\ell$. A cellular map $\phi:\C\to\bb{R}^{\ell}$ is called natural if for every $1\leq i\leq\ell$ and every $(x_1,\dots,x_{i-1})\in\pi_{i-1}(\C)$, the function $\phi_{i}(x_1,\dots,x_{i-1},\cdot)$ is affine. 
\end{defn}
The following Lemma is straightforward.
\begin{lem}
Let $C$ be a cell (not necessarily integer) of length $\ell$, and let $\C$ be any integer cell with the same type as $C$. Then there exists a unique surjective natural cellular map $A^{\C,C}:\C\to C$. If $C$ is definable in a sharply o-minimal structure and has format $\cc{F}$ and degree $D$, then $A^{\C,C}$ has format $O_{F}(1)$ and degree $\comp$.
\end{lem}
\begin{defn}
A morphism $\phi:\F^{\ell}_{1}\to\F^{\ell}_{2}$ is called natural if for every $\C_{1}\in C(\F^{\ell}_{1})$, the restriction $\phi|_{\C_{1}}$ is natural.
\end{defn}
Given a cylindrical decomposition $\Phi$ of $I^{\ell}$, there exists a unique pair $(\F^{\ell},\phi)$ where $\phi:\F^{\ell}\to I^{\ell}$ is a natural morphism such that for every $\C\in C(\F^{\ell})$, the restriction $\phi|_{\C}$ is a homeomorphism between $\C$ and a cell of $\Phi$. The fort $\F^{\ell}$ will be called the type of $\Phi$. We will now decribe a series of bijections illustrating the equivalence of morphisms and cylindrical parametrizations.
\begin{enumerate}
    \item Cylindrical decompositions of $I^{\ell}$ of type $\F^{\ell}$ are in $1-1$ correspondence with natural morphisms $\F^{\ell}\to I^{\ell}$.
    \item Cylindrical parametrizations of cylindrical decompositions of $I^{\ell}$ of type $\F^{\ell}$ are in $1-1$ correspondence with morphisms $\F^{\ell}\to I^{\ell}$.
    \item Given a cylindrical decomposition $\Phi$ which corresponds to a natural morphism $\phi:\F^{\ell}\to I^{\ell}$, cylindrical parametrizations of refinements $\Phi'$ of $\Phi$ of type $\G^{\ell}$ are in $1-1$ correspondence with commutative triangles of morphisms of the kind \begin{equation*}
        \begin{tikzcd}
            \F^{\ell} \arrow[r,"\phi"] & I^{\ell} \\ 
            \G^{\ell} \arrow[u] \arrow[ur]
        \end{tikzcd}.
    \end{equation*}
\end{enumerate}
In particular, this means that the ordinary cylindrical decomposition theorem from o-minimality can be reformulated in terms of forts and natural morphisms. We will use the following  somewhat stronger formulation. 
\begin{defn}
Let $\F$ be a fort, and let $\{X_{\C,j}\}_{\C\in C(\F),\;j\in J}$ be a collection of definable subsets of $\F$ such that $X_{\C,j}\subset\C$. We say that a morphism $\phi:\widetilde{\F}\to\F$ is compatible with $\{X_{\C,j}\}_{\C\in C(\F),\;j\in J}$ if $\phi(\widetilde{\C})$ is compatible with every $X_{\C,j}$.
\end{defn}
\begin{prop}.
\label{prop:cyl_dec_forts}
Let $\F$ be a fort, and let $\{X_{\C,j}\}_{\C\in C(\F),\;j\in J}$ be a collection of definable subsets of $\F$ such that $X_{\C,j}\subset\C$. Then there exists a natural morphism $\phi:{\widetilde{\F}}\to\F$ compatible with $\{X_{\C,j}\}_{\C\in C(\F),\;j\in J}$. 

If in addition the structure is \so-minimal, and the sets $X_{C,j}$ have format $\cc{F}$ and degree $D$, then $\widetilde{\F}$ can be taken to have size $\poly_{\cc{F}}(D,|C(\F)|,|J|)$ and the morphism $\phi$ can be chosen so that for every $\C\in C(\F)$ the restriction $\phi|_{\C}$ has format $\format$ and degree $\comp$.
\end{prop}
We postpone the proof of Proposition \ref{prop:cyl_dec_forts} until we develop the theory of forts further.
\subsection{Combinatorial Equivalence}
Here is a general useful lemma about morphisms.
\begin{lem}
Let $\phi:\F^{\ell}_{1}\to\F^{\ell}_{2}$ be a morphism. Then $\phi$ is one-to-one, and moreover for every $1\leq i\leq\ell$ and every $(x_1,\dots,x_{i-1})\in\pi_{i-1}(\F^{\ell}_{1})$ the function $\phi_i(x_1,\dots,x_{i-1},\cdot)$ is continuous. 
\end{lem}
\begin{proof}
We prove it by induction on $\ell$. For $\ell=1$, since $\phi$ is a strictly monotone increasing bijection between two intervals, it is clearly one-to-one and continuous. \\ 

Suppose the claim is true for $\ell-1$. Fix 
$(x_1,\dots,x_{\ell-1})\in\pi_{\ell-1}(\F^{\ell}_{1})$.The function $\phi_{\ell}(x_1,\dots,x_{\ell-1},\cdot)$ is continuous by Example \ref{ex:fibre_morphism} and the $\ell=1$ case. If $\phi(x)=\phi(x')$, we know that $\pi_{\ell-1}(x)=\pi_{\ell-1}(x')$ by induction, and then $x=x'$ follows from the fact that $\phi_{\ell}(x_1,\dots,x_{\ell-1},\cdot)$ is strictly monotone. 
\end{proof}
Suppose one has a definable family $\{\phi_{\lambda}:\F\to\G\}_{\lambda\in I}$ of morphisms from the fort $\F$ to the fort $\G$. Then the map $I\times\F\to I\times\G$ defined by $(\lambda,x)\mapsto(\lambda,\phi_{\lambda}(x))$ is an onto precellular map which is not necessarily a morphism. The first obstruction is the continuity in $\lambda$ of the restrictions $\phi_{\lambda}|_{\C}$ for all cells $\C\in C(\F)$. To formulate the second obstruction, we introduce the following definition.
\begin{defn}
Let $\phi,\psi:\F\to\G$ be morphisms. We say that $\phi,\psi$ are \emph{combinatorially equivalent} if for every $\C\in C(\F)$, the images $\phi(\C),\psi(\C)$ are contained in the same cell of $\G$.
\end{defn}
\begin{ex}
Let $\F\in \textnormal{Forts}(\ell)$, then all morphisms $\F\to I^{\ell}$ are combinatorially equivalent. 
\end{ex}
\begin{ex}
Consider the following two natural morphisms $\phi,\psi:(0,3)\to(0,2)$.
\begin{equation}
    \phi=\begin{cases}
    \frac{x}{2},\;x\in(0,2],\\
     x-1,\;x\in(2,3)
    \end{cases}
    \psi=\begin{cases}
    x,\;x\in(0,1],\\
    \frac{x+1}{2},\;x\in(1,3)
    \end{cases}
\end{equation} They are not combinatorially equivalent, since $\phi(1)=\frac{1}{2}$ and so $\phi$ maps the cell $\{1\}$ into the cell $(0,1)$, but $\psi(1)=1$, and so maps the cell $\{1\}$ into itself. 

\end{ex}
\begin{prop}
\label{prop:family_of_combinatorially_equivalent_morphisms_is_morphism}
Let $\{\phi_{\lambda}:\F\to\G\}_{\lambda\in I}$ be a definable family of combinatorially equivalent morphisms. Suppose further that for every $\C\in C(\F)$, the map $(\lambda,x)\mapsto\phi_{\lambda}(x)$ is continuous on $I\times\C$. Then the map $\psi:I\times\F\to I\times\G$ defined by $(\lambda,x)\to(\lambda,\phi_{\lambda}(x))$ is a morphism.
\end{prop}
\begin{proof}
Clearly, $\psi$ is onto and precellular. By the second assumpstion, we also know that $\psi$ is continuous on the cells of $I\times\F$. Now let $\C\in C(I\times\F)$ be a cell. Then there exists $\C'\in C(\F)$ such that $\C=I\times\C'$. Since the $\phi_{\lambda}$ are combinatorially equivalent, there exists $\C''\in C(\G)$ such that $\phi_{\lambda}(\C')\subset\C''$ for all $\lambda$. Thus, $\psi$  maps $I\times\C'$ into $I\times\C''$, as needed.
\end{proof}
The following is an important structural proposition on forts and morphisms.
\begin{prop}
\label{prop:structure}
Let $\phi:\F^{\ell}\to\G^{\ell}$ be a morphism, and let $\C\in C(\pi_{i}(\F^{\ell}))$ for some $1\leq i\leq\ell$. \begin{enumerate}
    \item For every $x,x'\in\C$ we have $\F^{\ell}_{x}=\F^{\ell}_{x'}$. Thus the notation $\F^{\ell}({\C}):=\F^{\ell}_{x}$ is justified.
    \item The map $C(\F^{\ell}({\C}))\to C(\F^{\ell})$ defined by $\C'\mapsto\C\times\C'$ is a bijection between $C(\F^{\ell}({\C}))$ and the cells $\widetilde{\C'}$ of $\F^{\ell}$ satisfying $\pi_{i}(\widetilde{\C'})=\C$.
    \item For every $x,x'\in\C$ the morphisms $\phi_{x},\phi_{x'}$ are combinatorially equivalent. 
\end{enumerate}
\end{prop}
\begin{proof} We prove the first item by ascending induction on $i$. \\
$1.$ For $i=1$, it is clear from Definition \ref{defn:fortsx}. Now suppose the claim is true for $i-1$. Let $(x_{i+1},\dots,x_{\ell})\in\F^{\ell}_{x}$. Since $x_{1\dots i-1},x^{'}_{1\dots i-1}$ are in the same cell of $\pi_{i-1}(\F^{\ell})$, by induction we have $\F^{\ell}_{x_{1\dots i-1}}=\F^{\ell}_{x'_{1\dots i-1}}$, thus $(x_{i},x_{i+1},\dots,x_{\ell})\in\F^{\ell}_{x'_{1\dots i-1}}$, or in other words $(x_{i+1},\dots,x_{\ell})\in\left(\F^{\ell}_{x'_{1\dots i-1}}\right)_{x_i}$. We claim that $x_{i},x'_{i}$ are in the same cell of $\pi_{1}\left(\F^{\ell}_{x'_{1\dots i-1}}\right)$. Indeed, since $x,x'$ are in the same cell of $\pi_{i}(\F^{\ell})$, we have that $x_{i},x'_{i}$ are bounded between the same two consecutive integers, or are both equal to an integer. This precisely means that they are in the same cell of $\pi_{1}(\F^{\ell}_{x'_{1\dots i-1}})$. So, again by the induction hypothesis, we have $\left(\F^{\ell}_{x'_{1\dots i-1}}\right)_{x_i}=\left(\F^{\ell}_{x'_{1\dots i-1}}\right)_{x'_i}$. We conclude that $(x_{i+1},\dots,x_{\ell})\in\left(\F^{\ell}_{x'_{1\dots i-1}}\right)_{x'_i}=\F^{\ell}_{x'_{1\dots i}}$. So $\F^{\ell}_{x}\subset\F^{\ell}_{x'}$, and by symmetry we even have equality. \\ \\ 
$2.$ Since the cells of $\F^{\ell}({\C})$ are disjoint, obviously this map is one to one. Let $\C'\in C(\F^{\ell}({\C}))$, clearly, $\C\times\C'$ is not disjoint from $\F^{\ell}$, and since it is an integer cell, it then follows that $\C\times\C'\in C(\F^{\ell})$, and obviously $\pi_{i}(\C\times\C')=\C$. Finally, let $\widetilde{\C'}\in\C(\F^{\ell})$ such that $\pi_{i}(\widetilde{\C'})=\C$. Then $\widetilde{\C'}=\C\times\C'$ for some integer cell $\C'$. Again, clearly $\C'$ is not disjoint from $\F^{\ell}({\C})$, and thus $\C'\in C(\F^{\ell}({\C}))$, and so the map is surjective. \\ \\ 
$3.$ Let $\C'\in C(\F^{\ell}({\C}))$, and let $\widetilde{\C''}\in C(\G^{\ell})$ be the cell that contains $\phi(\C\times\C')$. By the second item $\widetilde{\C''}$ corresponds to a cell $\C''\in C(\G^{\ell}_{\phi_{1\dots i}(x_{1\dots i})})$. We claim that $\phi_{x}(\C'),\phi_{x'}(\C')\subset\C''$. Indeed, let $y\in\C'$. Then $(x,y),(x',y)\in\C\times\C'$, and so $\phi(x,y),\phi(x',y)\in\widetilde{\C''}$. Again by the second item, this exactly means that $\phi_{x}(y),\phi_{x'}(y)\in\C''$.
\end{proof}
\subsection{Inverse image fort and proof of Proposition \ref{prop:cyl_dec_forts}}
\begin{lem}
Let $\F,\G\in \textnormal{Forts}(\ell)$, and let $\phi:\F\to\G$ be a morphism. Let $\G'\subset\G$ be a fort. Then $\F':=\phi^{-1}(\G')$ is a fort and $\phi|_{\F'}:\F'\to\G'$ is a morphism.
\end{lem}
\begin{proof}
If $\F'$ is a fort, then the second part of the lemma is obvious. We prove that it is a fort by induction on $\ell$. The case $\ell=1$ is clear. By the iductive hypothesis, $\phi^{-1}_{1\dots\ell-1}(\pi_{\ell-1}(\G'))$ is a fort, and since $\phi$ is precellular, we have $\pi_{\ell-1}(\F')=\phi^{-1}_{1\dots\ell-1}(\pi_{\ell-1}(\G'))$. Let $\C$ be one of the cells of $\phi^{-1}_{1\dots\ell-1}(\pi_{\ell-1}(\G'))$, and let $\C''$ be the cell of $\pi_{\ell-1}(\G')$ that contains $\phi_{1\dots\ell-1}(\C)$. Then $\G'(\C'')$ is naturally a subfort of $\G(\C'')$, and so by the induction hypothesis for any $x\in\C$ we have that $\F'_{x}$ is a fort. Moreover, according to proposition \ref{prop:structure}, for any $x,x'\in\C$, the morphisms $\phi_{x},\phi_{x'}:\F(\C)\to\G(\C'')$ are combinatorially equivalent, which means precisely that $\F'_{x}=\F'_{x'}$ for any such $x,x'$. We've shown that $\F'$ is a fort by Definition \ref{defn:fortsi}.
\end{proof}
\begin{proof}[Proof of Proposition \ref{prop:cyl_dec_forts}]
The fort $\F$ is contained in a fort $\fort{D}$ which is a linear subdivision of $(0,1)^{\ell}$ of order $d\leq|C(\F)|$. By \s CD there exists a cylindrical decomposition of $\fort{D}$ compatible with the sets $X_{\C,j}$ and the cells $\cc{P}\in C(\fort{D})$ into $\poly_{\cc{F}}(D,|C(\fort{D})|,|\{X_{\C,j}\}|)=\poly_{\cc{F}}(D,|C(\F)|,|J|)$ cells of format $\format$ and degree $\comp$. Indeed, an integer cell of length $\ell$ always has format $O_{\cc{F}}(1)$ and degree $\poly(\ell)=\poly(\cc{F})$ which is within our scale since $\poly_{\cc{F}}(D,\cc{F})=\comp$ and moreover $|C(\fort{D})|=\poly_{\ell}(|C(\F)|)$. \\ \\ 
In other words, there exists a surjective precellular map $\phi:\F'\to\fort{D}$, where $\F'$ is a fort with $|C(\F')|=\poly_{\cc{F}}(D,|C(\F)|,|J|)$, such that for any $\C'\in C(\F')$, $\phi|_{\C'}$ is continuous and has format $\format$ and degree $\comp$. Moreover, since $\phi(\C')$ is compatible with the cells of $\fort{D}$ which are disjoint, it follows that there exists a cell of $\fort{D}$ which contains $\phi(\C')$. Thus, $\phi$ is morphism. Denote $\widetilde{\F}:=\phi^{-1}(\F)$. It is now easy to see that $\phi|_{\widetilde{\F}}:\widetilde{\F}\to\F$ satisfies the requirements of the proposition.
\end{proof}
\subsection{Pulling back along extensions}
The pullback of a fort $\F^{\ell}$ along a morphism to the fort $\pi_{\ell-1}(\F^{\ell})$ is a useful natural construction. 
\begin{defn}
Let $\phi:\F^{\ell-1}_{1}\to \F^{\ell-1}_{2}$ be a morphism, and let $\F^{\ell}_{2}$ extend $\F^{\ell-1}_{2}$. We define the fort $\phi^{*}\F^{\ell}$. Suppose $\F^{\ell}_{2}=\F^{\ell-1}_{2}\odot\varphi$ (as in Definition \ref{defn:fortsi}), then we define $\phi^{*}\F^{\ell}_{2}:=\F_{1}^{\ell-1}\odot(\varphi\circ\phi)$.
The morphism $\phi$ extends to the morphism $(\phi,Id):\phi^{*}\F^{\ell}_{2}\to\F^{\ell}_{2}$.\\ \\ If one has a morphism $\phi:\F^{k}_{1}\to\F^{k}_{2}$ and for $1\leq k<\ell$ and a fort $\F^{\ell}_{2}$ extending $\F^{k}_{2}$, then $\phi^{*}\F^{\ell}_{2}\in \textnormal{Forts}(\ell)$ is defined by induction, and $\phi$ will again extend to a morphism $\phi^{*}\F^{\ell}_{2}\to\F^{\ell}_{2}$. 
\end{defn}
We will need the following lemma, which we leave as an exercise for the reader. 
\begin{lem}
\label{lem:combintorial_equivalence_same_pullback}
Let $\phi,\psi:\fort{F}^{k}\to\fort{G}^{k}$ be combinatorially equivalent, and let $\fort{G}^{\ell}$ extend $\fort{G}^{k}$. Then $\phi^{*}\fort{G}^{\ell}=\psi^{*}\fort{G}^{\ell}$, and moreover the extensions of $\phi,\psi$ to $\phi^{*}\G^{\ell}$ are combinatorially equivalent. 
\end{lem}
\subsection{The Tower construction}
Given a morphism $\phi:\F\to\G$, it will be more convenient to keep track of the formats and degrees of the restrictions of $\phi$ to the integer cells of $\F$, rather than the format and the degree of $\phi$ itself. We therefore introduce the following useful slight abuse of notation.
\begin{defn}
\label{defn:abuse}
We say that the family  $\{\phi^{\lambda}:\F\to\G\}_{\lambda\in\Lambda}$ of morphisms is an $(\cc{F},D)$-family of morphisms if for every $\C\in C(\F)$ the family $\{\phi^{\lambda}|_{\C}\}_{\lambda\in\Lambda}$ is an $(\cc{F},D)$-family as in Definition \ref{defn:format_degree_of_family}. 
\end{defn}
\begin{rem}
Let $\phi$ be any morphism, then in particular, the format and degree of $\phi$ as a morphism are different from its format and degree as a map. 
\end{rem}
If $\C$ is an integer cell, let $b(\C)$ be the basic cell with the same type as $\C$. If $m\in\bb{Z}$, denote $t_{m}(x_1,\dots,x_{\ell}):=(x_1+m,x_2,\dots,x_{\ell}).$

The point of Proposition \ref{prop:tower_construction} below is, that in order to construct a morphism to a fort $\F^{1}\ltimes s$, it is enough to construct for each $\C\in C(\F^{1})$ a morphism to $b(\C)\times s(\C)$. 
\begin{prop}[The tower construction]
\label{prop:tower_construction}
Let $\F^{1}\ltimes s$ be a fort of length $\ell$, and let $\C_1,\dots,\C_n$ be the cells of $\F^{1}$, ordered compatibly with their order in $\bb{R}$. Suppose that for every $1\leq i\leq n$ one has a morphism $\phi_{\C_i}:\F^{\ell}_{\C_i}\to b(\C_i)\times s(\C_i)$ of format $\cc{F}$ and degree $D$. Let $m_i:=\sup \pi_1(\F^{\ell}_{\C_i})$, and $m_0=0$. Then \begin{equation}
    \F^{\ell}:=\bigcup^{n-1}_{i=0}t_{m_0+\dots+m_i}(\F^{\ell}_{\C_{i+1}})
\end{equation}
is a fort, and the map $\phi:\F^{\ell}\to\F^{1}\ltimes s$ defined by $\phi|_{t_{m_0+\dots+m_i}(\F^{\ell}_{\C_{i+1}})}:=\phi_{\C_i}\circ t_{-(m_0+\dots+m_i)}$ is a morphism of format $\format$ and degree $\poly_{\cc{F}}(D)$.
\end{prop}
The proof is immediate. Note that the derivatives of $\phi$ are equal to the derivatives of $\phi_{\C_i}$, which will be crucial for constructing morphisms with bounded derivatives.
\subsection{Smoothing}We say that a morphism of forts is a $C^{r}$-morphism if its restriction to every integer cell is $C^{r}$.
We use Proposition \ref{prop:Cr_locus_effective} in the context of forts in the following way. 
\begin{prop}
\label{prop:smoothing}
\phantom{o}\\
$Sm_{\ell}$: Let 
 $\phi:\F^{\ell}\to\G^{\ell}$ be a natural morphism of format $\cc{F}$ and degree $D$, and fix $r\in\bb{N}$. Then there exists a fort $\fort{K}^{\ell}$ of size $\poly_{\max\{\cc{F},r\}}(D,|C(\F|))$ and a natural $C^{r}$-morphism $\psi:\fort{K}^{\ell}\to\F^{\ell}$ of format $\rformat$ and degree $\rcomp$ such that $\phi\circ\psi$ is a $C^{r}$-morphism.\\ \\
 $Fsm_{\ell}$: Let $\F^{\ell}$ be a fort and for every cell $\C\in C(\F)$ let $F_{C}=\{f_{\C,j}:\C\to I\}_{j\in J}$ be a collection of definable functions of format $\cc{F}$ and degree $D$. Then there exists a natural $C^{r}$-morphism $\phi:\fort{K}^{\ell}\to\F^{\ell}$ of format $\rformat$ and degree $\rcomp$ where $|C(\fort{K}^{\ell})|=\poly_{\max\{\cc{F},r\}}(D,|\F|,|J|)$ such that for every $\C'\in C(\fort{K}^{\ell})$ and every $\C\in C(\F)$ with $\phi(\C')\subset\C$, the function $\phi|_{\C'}^{*}f_{\C,j}$ is $C^{r}$ for every $j\in J$.
\end{prop}
\begin{proof}
The proof is by an induction similar in structure to the induction proof of the main result. That is, we prove $FSm_{1},FSm_{\ell-1}\to Sm_{\ell}, Sm_{\ell}\to FSm_{\ell}$. \\ \\
$FSm_{1}$: By the tower construction we may assume that $\F^{1}=I$, and we omit $\C$ from the notation $f_{\C,j}$. According to Proposition \ref{prop:Cr_locus_effective} each $f_{j}$ is $C^{r}$ outside $\rcomp$ points. Therefore all the $f_{j}$ are $C^{r}$ outside $|J|\cdot\rcomp$ points. We finish by applying Propositon \ref{prop:cyl_dec_forts} to these points, and note that any natural morphism of length $1$ is automatically $C^{r}$. \\ \\ 
$FSm_{\ell-1}\to Sm_{\ell}$: Let $\C\in C(\F^{\ell})$ be a cell. Since $\phi$ is natural, if $\textnormal{type}(\C)=(\dots,1)$ then there are $f_{\C},g_{\C}:\pi_{\ell-1}(C)\to\bb{R}$ such that
\begin{equation}
    \phi|_{\C}(x_1,\dots,x_{\ell})=(1-x_{\ell})f_{C}(x_1,\dots,x_{\ell-1})+x_{\ell}g_{C}(x_1,\dots,x_{\ell-1}).
\end{equation}
If $\textnormal{type}(\C)=(\dots,0)$, then $\phi_{\ell}$ can be identified with a function on $\C'=\pi_{\ell-1}(\C)$. Let $\F^{\ell-1}:=\pi_{\ell-1}(\F^{\ell})$ and for every $\C'\in C(\F^{\ell-1})$ define
\begin{equation}
\label{equation:F^(1)_of_smoothing}
\begin{split}
    F^{1}_{\C'}&:=\{f_{\C},g_{\C}:\;\C\in C(\F^{\ell}),\;\pi_{\ell-1}(\C)=\C',\;\textnormal{type}(\C)=(\dots,1)\},\\
    F^{2}_{\C'}&:=\{\phi_{\ell}|_{\C}:\;\C\in  C(\F^{\ell}),\;\pi_{\ell-1}(\C)=\C',\;\textnormal{type}(\C)=(\dots,0)\},\\
    F^{3}_{\C'}&:=\{\phi_{1}|_{\C'},\dots,\phi_{\ell-1}|_{\C'}\}.
\end{split}
\end{equation}
In other words, $F^{1}_{\C'}$ and $F^{2}_{\C'}$ correspond to the $\ell'th$ coordinate of $\phi$, while $F^{3}_{\C'}$  corresponds to the first $\ell-1$ coordinates. Now, apply $FSm_{\ell-1}$ to the fort $\F^{\ell-1}$ and $F_{\C'}=F^{1}_{\C'}\cup F^{2}_{\C'}\cup F^{3}_{\C'}$. We obtain a natural $C^{r}$-morphism $\psi':\fort{K}^{\ell-1}\to\F^{\ell-1}$. Since $F^{3}_{\C'}\subset F_{\C}$, the morphism $\phi_{1...\ell-1}\circ\psi':\fort{K}^{\ell-1}\to\pi_{\ell-1}(\G^{\ell})$ is a $C^{r}$-morphism. Now let $\fort{K}^{\ell}:=(\psi')^{*}\F^{\ell}$, so that $\psi'$ extends to a natural $C^{r}$-morphism $\psi:\fort{K}^{\ell}\to\F^{\ell}$. Finally, due to $F^{1}_{\C'},F^{2}_{\C'}\subset F_{\C'}$, we see that $\phi\circ\psi$ is $C^{r}$, as needed. It is easy to track the formats and degrees of this construction. \\ \\ 
$Sm_{\ell}\to FSm_{\ell}:$ Let $k$ be the smallest integer such that there exists a cell $\C$ with at least $k$ zeros in its type and a $j\in J$ such that $f_{\C,j}$ is not $C^{r}$ on $\C$. Due to Proposition \ref{prop:Cr_locus_effective}, for every $\C\in C(\F)$ there exists a  definable set $V_{\C}\subset\C$ of format $\rformat$ and degree $\poly_{\max\{\cc{F},r\}}(D,|J|)$ with $\dim(\C)-\dim(V_{\C})\geq 1$ such that outside it, $f_{\C,j}$ is $C^{r}$ for every $j\in J$. By Proposition \ref{prop:cyl_dec_forts} there is a natural morphism $\psi:\fort{G}^{\ell}\to\F^{\ell}$ compatible with the collection $\{V_{\C}\}_{\C\in C(\F^{\ell})}$. By $Sm_{\ell}$ we may assume that $\psi$ is a $C^{r}$-morphism. Thus we have increased $k$ by 1, we repeat this procedure until $k=\ell$, at which point we are done.
\end{proof}
\subsection{The main result reformulated}
A morphism is called an $r$-morphism if all its restrictions to integer cells are $r$-maps.
We are now ready to state the family version of our main result in terms of forts and morphisms. We first state the case without parameters for simplicity. The reader should keep Definition \ref{defn:abuse} in mind.
\begin{thm}\phantom{o}\\
$S_{\ell,0}$: Let $\phi:\F^{\ell}\to\G^{\ell}$ be a natural morphism of format $\cc{F}$ and degree $D$. Then there exists a fort $\mathscr{K}^{\ell}$ with $|C(\mathscr{K^{\ell}})|=\poly_{\cc{F},r}(D,|C(\F^{\ell})|)$ and an $r$-morphism $\psi:\mathscr{K}^{\ell}\to\F^{\ell}$ of format $O_{\cc{F}}(1)$ and degree $\poly_{\cc{F}}(D,r)$ such that the composition $\phi\circ\psi$ is an $r$-morphism. \\ \\ 
$F_{\ell,0}$: Let $\F^{\ell}$ be a fort and for every $\C\in C(\F^{\ell})$ let $F_{\C}=\{f_{\C,j}:\C\to I\}_{j\in J}$ be a collection of definable functions such that each $f_{\C,j}$ has format $\cc{F}$ and degree $D$. Then there exists a fort $\mathscr{K}^{\ell}$ with $|C(\mathscr{K})|=\poly_{\max\{\cc{F},r\}}(D,|C(\F^{\ell})|)$ and an $r$-morphism $\phi:\mathscr{K}^{\ell}\to\F^{\ell}$ of format $O_{\cc{F},r}(1)$ and degree $\poly_{\cc{F},r}(D)$, such that for every $\C\in C(\F^{\ell}),j\in J$ and every $\C'\in C(\mathscr{K})$ with $\phi(\C')\subset\C$, the function $(\phi|_{\C'})^{*}f_{\C,j}$ is an $r$-function.
\end{thm}
$S_{\ell,0}$ implies $S^{*}_{\ell}$ of Theorem \ref{thm:mris} in the following way. Let $\Phi$ be a cylindrical decomposition of $I^{\ell}$ of size $N$ into cells of format $\cc{F}$ and degree $D$. Then $\Phi$ corresponds to a natural morphism $\phi:\F^{\ell}\to I^{\ell}$ such that $|C(\F^{\ell})|=N$ and the restrictions $\phi|_{\C}$ are of format $O_{\cc{F}}(1)$ and degree $\comp$ for all $\C\in C(\F^{\ell})$. Let $\psi:\mathscr{K}^{\ell}\to\F^{\ell}$ be the morphism whose existence is guaranteed by $S_{\ell,0}$ for $\phi$. Then the composition $\phi\circ\psi:\mathscr{K}^{\ell}\to I^{\ell}$ corresponds to cylindrical $r$-parametrization of a refinment $\Phi'$ of $\Phi$ that has size $|C(\mathscr{K})|=\poly_{\max\{\cc{F},r\}}(D,|C(\F^{\ell})|)=\poly_{\cc{F},r}(D,N)$. Moreover, for every $\C'\in\mathscr{K}$, the restriction $(\phi\circ\psi)|_{\C'}$ is a composition of two maps that have format $\format$, and have degrees $\comp,\rcomp$. Therefore by Lemma \ref{lem:sharp_arithmetic}, $(\phi\circ\psi)|_{\C'}$ has format $\format$ and degree $\rcomp$, as needed. $F_{\ell,0}$ implies $F^{*}_{\ell}$ in a similar way.\\ \\ 

We now state the family version of our main result.
\begin{thm}
\label{thm:the_main_result_of_this_paper}
\phantom{o}\\
$S_{\ell,k}$: Let $\{\phi^{\lambda}:\F^{\ell}\to\G^{\ell}\}_{\lambda\in I^{k}}$ be an $(\cc{F},D)$-family of natural morphisms. Then there exists
\begin{itemize}
    \item A natural morphism $\varphi:\fort{K}^{k}\to I^{k}$ of format $\rformat$, degree $\rcomp$ where $|C(\fort{K}^{k})|=\poly_{\max\{\cc{F},r\}}(D,|C(\F^{\ell})|)$,
    \item For every $\cc{P}\in C(\fort{K})$ a fort  $\mathscr{K_{\cc{P}}}$ with $|C(\mathscr{K}_{\cc{P}})|=\poly_{\max\{\cc{F},r\}}(D,|C(\F^{\ell})|)$ and an $(\rformat,\rcomp)$-family of combinatorially equivalent $r$-morphisms $\{\psi^{\lambda}:\mathscr{K}_{\cc{P}}\to\F^{\ell}\}_{\lambda\in\varphi(\cc{P})}$,
\end{itemize}
such that for every $\lambda\in\varphi(\cc{P})$ the composition $\phi^{\lambda}\circ\psi^{\lambda}$ is an $r$-morphism. \\ \\ 
$F_{\ell,k}$: Let $\F^{\ell}$ be a fort and for every $\C\in C(\F^{\ell})$ let $F_{\C,\lambda}=\{f_{\C,j,\lambda}:\C\to I\}_{\lambda\in I^{k},j\in J}$ be a collection of $|J|$ families such that for every fixed $j\in J$ the family $\{f_{\C,j,\lambda}:\C\to I\}_{\lambda\in I^{k}}$ is an $(\cc{F},D)$-family of definable functions. Then there exists
\begin{itemize}
    \item A natural morphism $\varphi:\fort{K}^{k}\to I^{k}$ of format $\rformat$, degree $\poly_{\cc{F},r}(D)$ where $|C(\fort{K}^{k})|=\poly_{\max\{\cc{F},r\}}(D,|C(\F^{\ell})|,|J|)$,
    \item For every $\cc{P}\in C(\fort{K})$ a fort $\mathscr{K_{\cc{P}}}$ with $|C(\mathscr{K}_{\cc{P}})|=\poly_{\cc{F}}(D,|C(\F^{\ell})|,|J|)$ and an $(\rformat,\rcomp)$-family of combinatorially equivalent $r$-morphisms $\{\phi^{\lambda}:\mathscr{K}_{\cc{P}}\to\F^{\ell}\}_{\lambda\in\varphi(\cc{P})}$,
\end{itemize}
such that for every $\lambda\in\varphi(\cc{P})$ and every $f_{\C,j,\lambda}\in F_{\C,\lambda}$ where $\phi^{\lambda}(\C')\subset\C$, the pullback $\left(\phi^{\lambda}|_{\C'}\right)^{*}f_{\C,j,\lambda}$ is an $r$-function.
\end{thm}
\subsection{Final Remarks}
\begin{rem}
It is easy to notice that composing a morphism $\phi$ with linear subdivision of order $d$ decreases the partial derivatives of $\phi$ by a factor of at least $d^{-1}$. Similarly if $\phi:\F\to\G$ is a morphism and $f$ is a function on some cell of $\G$, then the $d$-subdivision of $\F$ will also decrease the partial derivatives of $\phi^{*}f$ by a factor of at least $d$. Linear subdivision of order $d$ increases the size of a fort of length $\ell$ by a factor of $d^{\ell}$, and quite often we shall use linear subdivision with $d=\rcomp$ or $d=\rformat$, we shall use this reduction freely. For example, the composition of two $r$-morphisms is an $r$-morphism up to linear subdivision of order $O_{\ell,r}(1)=\rformat$.
\end{rem}
\begin{rem}
\label{rem:legal_reduction}
Let $\F^{\ell}$ be a fort and for every $\C\in C(\F^{\ell})$ let $F_{\C,\lambda}=\{f_{\C,j,\lambda}:\C\to I\}_{\{\lambda\in I^{k},j\in J\}}$ be a collection of $|J|$ families such that for every fixed $j\in J$ the family $\{f_{\C,j,\lambda}:\C\to I\}_{\lambda\in I^{k}}$ is an $(\cc{F},D)$-family of definable functions. Suppose we wish to find
\begin{itemize}
    \item A natural morphism $\varphi:\fort{K}^{k}\to I^{k}$ of format $\rformat$, degree $\poly_{\cc{F},r}(D)$ where $|C(\fort{K}^{k})|=\poly_{\max\{\cc{F},r\}}(D,|C(\F^{\ell})|,|J|)$,
    \item For every $\cc{P}\in C(\fort{K})$ a fort $\mathscr{K_{\cc{P}}}$ with $|C(\mathscr{K}_{\cc{P}})|=\poly_{\cc{F}}(D,|C(\F^{\ell})|,|J|)$ and an $(\rformat,\rcomp)$-family of combinatorially equivalent $r$-morphisms $\{\phi^{\lambda}:\mathscr{K}_{\cc{P}}\to\F^{\ell}\}_{\lambda\in\varphi(\cc{P})}$,
\end{itemize}
such that for every $\lambda\in\varphi(\cc{P})$ and every $f_{\C,j,\lambda}\in F_{\C,\lambda}$ where $\phi^{\lambda}(\C')\subset\C$, the pullback $\left(\phi^{\lambda}|_{\C'}\right)^{*}f_{\C,j,\lambda}$ has a property $P$. Suppose moreover that this can be done if the functions $f_{\C,j,\lambda}$ have the property $Q$. Then it is sufficient to find 
\begin{itemize}
    \item A natural morphism $\varphi':\fort{K}'^{k}\to I^{k}$ of format $\rformat$, degree $\poly_{\cc{F},r}(D)$ where $|C(\fort{K}^{k})|=\poly_{\max\{\cc{F},r\}}(D,|C(\F^{\ell})|,|J|)$,
    \item For every $\cc{P}\in C(\fort{K}'^{k})$ a fort $\mathscr{K_{\cc{P}}}'$ with $|C(\mathscr{K}_{\cc{P}}')|=\poly_{\cc{F}}(D,|C(\F^{\ell})|,|J|)$ and an $(\rformat,\rcomp)$-family of combinatorially equivalent $r$-morphisms $\{\phi'^{\lambda}:\mathscr{K}'_{\cc{P}}\to\F^{\ell}\}_{\lambda\in\varphi(\cc{P})}$,
\end{itemize}
such that for every $\lambda\in\varphi(\cc{P})$ and every $f_{\C,j,\lambda}\in F_{\C,\lambda}$ where $\phi^{\lambda}(\C')\subset\C$, the pullback $\left(\phi'^{\lambda}|_{\C'}\right)^{*}f_{\C,j,\lambda}$ has the property $Q$.
 Indeed, we first decompose $I^{k}$ into cells $\cc{P}$ on  which the functions $f_{\C,j,\lambda}$ have the property $Q$, and then we decompose $\cc{P}$ into cells on which the functions $f_{\C,j,\lambda}$ have the property $P$. The upper bounds on format and degree and the size of the forts above follow from Lemma \ref{lem:makecyl} and linear subdivision.
 
 In particular, we will prove $F_{\ell,k}$ step by step, at each step reducing to the case where the functions $f_{\C,j,\lambda}$ have an additional property.
 \end{rem}
\section{The step \texorpdfstring{$F_{1,k}$}{Lg}}
We will need a series of lemmas.
\begin{lem}
\label{lem:smoothing_a_family}
Let $\{f_{1,\lambda}:I\to I\}_{\lambda\in I^{k}},\dots,\{f_{s,\lambda}:I\to I\}_{\lambda\in I^{k}}$ be $(\cc{F},D)$-families, and let $r\in\bb{N}$. Then there exists\begin{itemize}
    \item A natural morphism $\varphi:\fort{K}^{k}\to I^{k}$ of format $\rformat$, degree $\poly_{\cc{F},r}(D,s)$ where $|C(\fort{K}^{k})|=\poly_{\cc{F},r}(D,s)$,
    \item For every $\cc{P}\in C(\fort{K}^{k})$ a fort $\fort{K}^{1}_{\cc{P}}$ with $|C(\fort{K}^{1}_{\cc{P}})|=\poly_{\cc{F},r}(D,s)$ and an $(\rformat,\rcomp)$-family of combinatorially equivalent natural morphisms $\{\phi^{\lambda}:\fort{K}^{1}_{\cc{P}}\to I\}_{\lambda\in\varphi(\cc{P})}$,
\end{itemize} 
such that for every $\lambda\in\varphi(\cc{P})$ and every $\C\in C(\fort{K}_{\cc{P}})$ the pullbacks $(\phi^{\lambda}|_{\C})^{*}f_{i,\lambda}$ are $C^{r}$-functions. 
\end{lem}
\begin{proof}
By Proposition \ref{prop:smoothing} there exists a natural morphism $\varphi':\fort{K}^{k+1}\to I^{k+1}$ such that for every $\cc{P}'\in C(\fort{K}^{k+1})$ and every $1\leq j\leq s$, the pullbacks $(\phi'|_{\cc{P}'})^{*}f_{j,\lambda}(x)$ are $C^{r}$ as functions of $(\lambda,x)$. Denote $\fort{K}^{k}=\pi_{k}(\fort{K}^{k+1})$, $\varphi=\varphi'_{1\dots k}$ and for every cell $\cc{P}\in C(\fort{K}^{k})$ define $\fort{K}^{1}_{\cc{P}}:=\fort{K}^{k+1}(\cc{P})$. Then for every $\lambda\in\varphi(\cc{P})$ the morphism $\varphi'$ induces a morphism $\phi^{\lambda}:=\varphi'_{\lambda}:\fort{K}^{1}_{\cc{P}}\to I$, and these morphisms are combinatorially equivalent by Proposition \ref{prop:structure}, this finishes the proof.
\end{proof}
The following lemma is a family version of the original argument by Gromov in \cite{Gromov}.
\begin{lem}
\label{lem:Gromov_lemma_for_forts}
Suppose $r\geq 2$, and
let $\{f_{1,\lambda}:I\to I\}_{_\lambda\in I^{k}},\dots,\{f_{s,\lambda}:I\to I\}_{_\lambda\in I^{k}}$ be $(\cc{F},D)$-families of $(r-1)$-functions. Then there exists\begin{itemize}
    \item A natural morphism $\varphi:\fort{K}^{k}\to I^{k}$ of format $\rformat$,degree $\rcomp$ where $|C(\fort{K}^{k})|=\poly_{\max\{\cc{F},r\}}(1)(D,s)$,
    \item For every $\cc{P}\in C(\fort{K}^{k})$ a fort $\fort{K}^{1}_{\cc{P}}$ with $|C(\fort{K}^{1}_{\cc{P}})|=\poly_{\cc{F},r}(D,s)$, and an $(\rformat,\rcomp)$-family $\{\phi^{\lambda}:\fort{K}_{\cc{P}}\to I\}_{\lambda\in\varphi(\cc{P})}$ of combinatorially equivalent $r$-morphisms,
\end{itemize}
such that for every $\C\in C(\fort{K}_{\cc{P}})$ and every $\lambda\in\varphi(\cc{P})$ the pullbacks $(\phi^{\lambda}|_{\C})^{*}f_{i,\lambda}$ are $r$-functions. 
\begin{proof}
Note that the $\phi^{\lambda}$ obtained in Lemma \ref{lem:smoothing_a_family} are natural, and this means that up to a linear division of order $d=O_{r}(1)$, by Lemma \ref{lem:smoothing_a_family}, the tower construction and Remark \ref{rem:legal_reduction}, we can assume that the functions $f_{i,\lambda}$ are $C^{r+1}$-functions.\\

Let the coordinates of $I^{k+1}$ be $(\lambda,x)$ and let $\varphi':\fort{K}^{k+1}\to I^{k+1}$ be a natural morphism compatible with the sets $\{f^{(r+1)}_{j,\lambda}=0\}$ and $\{f^{(r)}_{j,\lambda}=0\}$ for every $1\leq j\leq s$. Denote $\fort{K}^{k}=\pi_{k}(\fort{K}^{k+1})$, $\varphi=\varphi'_{1\dots k}$ and for every $\cc{P}\in C(\fort{K}^{k})$ define $\fort{K}^{1}_{\cc{P}}:=\fort{K}^{k+1}(\cc{P})$. For every $\lambda\in\varphi(\cc{P})$ the morphism $\varphi'$ induces a morphism $\phi^{\lambda}:=\varphi'_{\lambda}:\fort{K}^{1}_{\cc{P}}\to I$, and these morphisms are combinatorially equivalent by Proposition \ref{prop:structure}. Since $\phi^{\lambda}$ are natural, up to a linear subdivision of order $d=O_{r}(1)$, the tower construction and Remark \ref{rem:legal_reduction}, we have reduced to the case where the family $\{f^{(r)}_{j,\lambda}\}_{\lambda\in I^{k}}$ is either a family of monotone increasing functions or a family of monotone decreasing functions (or a family of constant functions, in which case we are done) of constant sign for every $1\leq j\leq s$. Moreover, since an $r$-parametrization of $f$ is the same as an $r$-parametrization of $-f$, we may additionally assume that $f^{(r)}_{i,\lambda}(x)\geq 0$ for all $x$ and all $\lambda$.\\

Define $\phi:I\to I$ by $x\mapsto 3x^{2}-2x^{3}$. It is an $r$-morphism up to linear subdivision of order $3$. If $f^{(r)}_{i,\lambda}$ is monotone decreasing, then \begin{equation}
\label{equation:decreasing_case}
\frac{2}{x}\geq\frac{f_{i,\lambda}^{(r-1)}(x)-f_{i,\lambda}^{(r-1)}(0)}{x}=f^{(r)}_{i,\lambda}(c_{x})\geq f^{(r)}_{i,\lambda}(x). 
\end{equation}
If $f^{(r)}_{i,\lambda}$ is monotone increasing, then 
\begin{equation}
\label{equation:increasing_case}
    \frac{2}{1-x}\geq\frac{f^{(r-1)}_{i,\lambda}(1)-f^{(r-1)}_{i,\lambda}(x)}{x}=f_{i,\lambda}^{(r)}(c_{x})\geq f^{(r)}_{i,\lambda}(x).
\end{equation}
We have \begin{equation}
    (f_{i,\lambda}(\phi(x)))^{(r)}=O_{r}(1)+O_{r}\left((\phi')^{r}f^{(r)}_{i,\lambda}(\phi(x))\right)
\end{equation} Note that $\phi'=6x(1-x)$.
If $f^{(r)}_{i,\lambda}$ is monotone decreasing then it is bounded near 1 by equation \ref{equation:decreasing_case}, and so $(\phi')^{r}f^{(r)}_{i,\lambda}(\phi(x))=O_{r}(1)$ near 1. Near $0$ we have $(\phi')^{r}f^{(r)}_{i,\lambda}(\phi(x))=O_{r}(x^{r-2})$ by equation \ref{equation:decreasing_case} and since $r\geq 2$ we get overall that $(f_{i,\lambda}(\phi(x)))^{(r)}=O_{r}(1)$. If $f^{(r)}_{i,\lambda}$ is monotone increasing then it is bounded near 0 by equation \ref{equation:increasing_case}, and so $(\phi')^{r}f^{(r)}_{i,\lambda}(\phi(x))=O_{r}(1)$ near $0$. Near $1$ we have $(\phi')^{r}f^{(r)}_{i,\lambda}(\phi(x))=O_{r}((1-x)^{r-2})$ by equation \ref{equation:increasing_case}, and again since $r\geq 2$ we have $(\phi')^{r}f^{(r)}_{i,\lambda}(\phi(x))=O_{r}(1)$. Thus we are done by linear subdivision of order $O_{r}(1)$.
\end{proof}
\end{lem}
\begin{proof}[Proof of $F_{1,k}$] 
The idea of the proof is as follows. By Lemma \ref{lem:Gromov_lemma_for_forts} above, it is sufficient to reduce to the case where the functions $f_{j,\lambda}$ are $1$-functions. Suppose we have definable $C^{r}$ functions $f_1,\dots,f_{s}:I\to I$, and suppose that $|f'_{1}|<\dots<|f'_{s}|$, and that either $|f'_{s}|\geq1$ or $|f'_{s}|\leq1$ globally on $I$. If $|f'_{s}|\leq 1$, we are done. If $|f'_{s}|\geq1$, we can reparametrize $I$ by $f^{-1}_{s}$, and then we will be done. We will now reduce to one of these two cases. \\ 

By the tower construction we may assume that $\F^{1}=I$. By Lemma \ref{lem:smoothing_a_family} we may further assume that the functions $f_{i,\lambda}(x)$ are in $C^{2}$. Now define 
\begin{equation}
\begin{split}
       g_{ij}(\lambda,x)&:=|f'_{i,\lambda}(x)|-|f'_{j,\lambda}(x)|,\\
       h_{i}(\lambda,x)&:=|f'_{i,\lambda}(x)|-1.
\end{split}
\end{equation}
Let the coordinates of $I^{k+1}$ be $(\lambda,x)$ and let $\varphi':\fort{K}^{k+1}\to I^{k+1}$ be a natural morphism compatible with the sets $\{g_{ij}=0\},\{h_{i}=0\}$ and $\{f'_{j,\lambda}(x)=0\}$ for every $1\leq i,j\leq s$. Denote $\varphi=\varphi'_{1\dots k}$, and for every $\cc{P}\in C(\fort{K}^{k})$ define $\fort{K}^{1}_{\cc{P}}:=\fort{K}^{k+1}(\cc{P})$. For every $\lambda\in\varphi(\cc{P})$ define $\Phi^{\lambda}_{\cc{P}}=\{\varphi_{\lambda}(\C):\;\C\in C(\fort{K}^{1}_{\cc{P}})\}$, a cylindrical decomposition of $I$ of size $\poly_{\cc{F}}(D,s)$.\\

We aim to construct a family of $r$-morphisms $\{\phi^{\lambda}:\fort{K}^{1}_{\cc{P}}\to I\}_{\lambda\in\varphi(\cc{P})}$ which will be a family of cylindrical $r$-parametrizations of $\Phi^{\lambda}_{\cc{P}}$, such that for every $\lambda\in\varphi(\cc{P})$, every $1\leq j\leq s$ and every $\C\in C(\fort{K}^{1}_{\cc{P}})$ the pullback $(\phi^{\lambda}|_{\C})^{*}f_{j,\lambda}$ is an $1$-function. To do this, for every interval $L_{\lambda}\in\Phi^{\lambda}$ we define a monotone increasing surjective $r$-map $\phi_{L_{\lambda}}:I\to L_{\lambda}$ such that for every $1\leq j\leq s$ the pullback $(\phi_{L_{\lambda}})^{*}f_{j,\lambda}|_{L_{\lambda}}$ is an $1$-function.\\

Fix $\cc{P}\in C(\fort{K}^{k})$, a $\lambda\in\varphi(\cc{P})$ and an interval $L_{\lambda}\in\Phi^{\lambda}_{\cc{P}}$. Since $\varphi'$ is compatible with $f'_{j,\lambda}$, on $L_{\lambda}$ some of the functions $f_{j,\lambda}$ are constant (and their indices $j$ are constant as $\lambda$ ranges over $\varphi(\cc{P})$) and those that aren't constant are strictly monotone. Constant functions are automatically $r$-functions, so we may assume that none of the $f_{j,\lambda}$ are constant. Since $\varphi'$ is compatible with $g_{ij}$, we can relabel the indices $j$ so that $|f'_{1,\lambda}(x)|<\dots<|f'_{s,\lambda}(x)|$ for every $\lambda\in\varphi(\cc{P})$ and every $x\in L_{\lambda}$. Since $\varphi'$ is compatible with $h_{j}$, we have that either $|f'_{s,\lambda}(x)|\leq 1$ for all $\lambda\in\varphi(\cc{P}),\;x\in L_{\lambda}$, or $|f'_{s,\lambda}(x)|>1$ for all $\lambda\in\varphi(\cc{P}),\;x\in L_{\lambda}$.\\

If $|f'_{s,\lambda}(x)|\leq1$, we simply parametrize the interval $L_{\lambda}$ by an affine map $\phi_{L_{\lambda}}:I\to L_{\lambda}$. Otherwise, suppose without loss of generality that $f_{s,\lambda}$ is monotone increasing. We parametrize $L_{\lambda}$ by the map $\phi_{L_{\lambda}}=f^{-1}_{s,\lambda}\circ A_{\lambda}:I\to L_{\lambda}$, where $A_{\lambda}:I\to f_{s,\lambda}(L_{\lambda})$ is an affine map. We claim that $|\phi_{L_{\lambda}}'|\leq 1$. Indeed, \begin{equation}
    |\phi'_{L_{\lambda}}(x)|=\frac{\length(f_{s,\lambda}(L_{\lambda}))}{|f'_{s,\lambda}(\phi_{L_{\lambda}}(x))|}\leq\frac{\length(I)}{1}=1.
\end{equation}
Finally, we also claim that $(\phi_{L_{\lambda}})^{*}f_{j,\lambda}$ is a $1$-function for every $1\leq j\leq s$. Indeed,\begin{equation}
    |((\phi_{L_{\lambda}})^{*}f_{j,\lambda})'(x)|=\length(f_{s,\lambda}(L_{\lambda}))\cdot\frac{|f'_{j,\lambda}(\phi_{L_{\lambda}}(x))|}{|f'_{s,\lambda}(\phi_{L_{\lambda}}(x))|}\leq 1
\end{equation} where the last inequality follows as $|f_{j,\lambda}'|<|f'_{s,\lambda}|$.
Thus we may assume that the functions $f_{j,\lambda}$ are $1$-functions. We are now in position to use Lemma \ref{lem:Gromov_lemma_for_forts} $r-1$ times, and we are done. 
\end{proof}
\section{The step \texorpdfstring{$S_{\ell-1,k}+F_{\ell-1,k}\to S_{\ell,k}$}{Lg}}
Let $\{\phi
^{\lambda}:\F^{\ell}\to\G^{\ell}\}_{\lambda\in I^{k}}$ be an $(\cc{F},D)$-family of natural morphisms. Fix a cell $\C\in\ C(\pi_{\ell-1}(\F^{\ell}))$ and a cell $\C'\in C(\F^{\ell})$ with $\pi_{\ell-1}(\C')=\C$.  Since $\phi^{\lambda}$ is natural,  if $\textnormal{type}(\C')=(\dots,1)$ then there are definable functions $f_{\C',\lambda},g_{\C',\lambda}$ on $\C$ such that \begin{equation}
    \phi^{\lambda}_{\ell}|_{\C'}=(1-x_{\ell})f_{\C',\lambda}(x_1,\dots,x_{\ell-1})+x_{\ell}g_{\C',\lambda}(x_1,\dots,x_{\ell-1}),
\end{equation}
and if $\textnormal{type}(\C')=(\dots,0)$ then $\phi^{\lambda}_{\ell}$ can be identified with a function on $\C$, and we maintain the notation $\phi^{\lambda}_{\ell}$.
Define \begin{equation}
\begin{split}
    F^{1}_{\C,\lambda}&:=\{f_{\C',\lambda},g_{\C',\lambda}:\C'\in C(\F^{\ell}),\;\pi_{\ell-1}(\C')=\C,\textnormal{type}(\C')=(\cdot,1)\},\\
    F^{2}_{\C,\lambda}&:=\{\phi^{\lambda}_{\ell}|_{\C'}:\C'\in C(\F^{\ell}),\;\pi_{\ell-1}(\C')=\C,\textnormal{type}(\C')=(\cdot,0)\},\\
    F^{3}_{\C,\lambda}&:=\{\phi^{\lambda}_{1}|_{\C},\dots,\phi^{\lambda}_{\ell-1}|_{\C}\},\\
    F_{\C,\lambda}&:=F^{1}_{\C,\lambda}\cup F^{2}_{\C,\lambda}\cup F^{3}_{\C,\lambda}.
\end{split}
\end{equation}
Lemma \ref{lem:sharp_arithmetic} implies that for every $\C'\in C(\F^{\ell})$ with $\pi_{\ell-1}(\C')=\C$ and $\textnormal{type}(\C')=(\cdot,1)$, the families $\{f_{\C',\lambda}\}_{\lambda\in I^{k}}$ and $\{g_{\C',\lambda}\}_{\lambda\in I^{k}}$ are $(\format,\comp)$-families. Now use $F_{\ell-1,k}$ on the fort $\F^{\ell-1}:=\pi_{\ell-1}(\F^{\ell})$ and the family of functions $F_{\C,\lambda}$.\\

We obtain a natural morphism $\varphi:\fort{K}^{k}\to I^{k}$, where $\fort{K}^{k}$ is a fort of size $|C(\fort{K}^{k})|=\poly_{\max\{\cc{F},r\}}(D,|C(\F^{\ell})|)$, for each cell $\cc{P}\in C(\fort{K}^{k})$ a fort $\fort{K}'_{\cc{P}}$ of size $\poly_{\max\{\cc{F},r\}}(D,|C(\F^{\ell})|)$ and an $(\rformat,\rcomp)$-family of combinatorially equivalent $r$-morphisms $\{\psi'^{\lambda}:\fort{K}'_{\cc{P}}\to\F^{\ell-1}\}_{\lambda\in\varphi(\cc{P})}$.\\ 

Since $F^{3}_{\C,\lambda}\subset F_{\C,\lambda}$ we see that $\phi^{\lambda}_{1...\ell-1}\circ\psi'^{\lambda}$ is an $r$-morphism. Define $\fort{K}_{\cc{P}}:=(\psi'^{\lambda})^{*}\F^{\ell}$. Since the $\psi'^{\lambda}$ are combinatorially equivalent, $\fort{K}_{\cc{P}}$ indeed does not depend on $\lambda$. $\psi'^{\lambda}$ extends to an $r$-morphism $\psi^{\lambda}:\fort{K}_{\cc{P}}\to\F^{\ell}$, and since $F^{1}_{\C,\lambda},F^{2}_{\C,\lambda}\subset F_{\C,\lambda}$ we see that $\phi^{\lambda}\circ\psi^{\lambda}$ is an $r$-morphism as needed.
\section{The step \texorpdfstring{$S_{\leq\ell,k}+F_{\leq\ell-1,k}\to F_{\ell,k}$}{Lg}}
We will need the following lemmas.
\begin{lem}
\label{lem:family_of_lipshitz}
Let $f:I^{\ell}\to I$ be a definable function of format $\cc{F}$ and degree $D$, and suppose that for all $x_{1}\in I$ the function $f(x_{1},\cdot)$ is $L$-Lipshitz. Then outside $\comp$ hyperplanes of the form $\{x_{1}=\textnormal{const}\}$ the function $f$ is continuous. 
\end{lem}
\begin{proof}
  Consider the set $\textnormal{Var}:=\{x:\overline{\lim}_{y\to x}|f(y)-f(x)|>0\}$. Obviously, $f$ is continuous on the set $I^{\ell}\backslash \textnormal{Var}$. We claim that for every $t\in I$, the intersection $\textnormal{Var}\cap\{x_{1}=t\}$ is open in the hyperplane $\{x_{1}=t\}$. Indeed, suppose that for some $x\in\{x_{1}=t\}$
  \begin{equation}\label{eq:eps definition blal}
  \epsilon=\overline{\lim}_{y\to x}|f(y)-f(x)|>0.
  \end{equation} 
  If $x'\in\{x_{1}=t\}$ and $|x-x'|<\frac{\epsilon}{4L}$ then \begin{equation}
  \begin{split}
      |f(y)-f(x')|&\geq |f(y-x'+x)-f(x)|-|f(x')-f(x)|-|f(y-x'+x)-f(y)|\geq\\
                    & \geq|f(y-x'+x)-f(x)|-\frac{\epsilon}{2},
  \end{split}
  \end{equation}
  since $y-x'+x,y$ also have the same $x_{1}$ coordinate. 
  Considering the $\overline{\lim}$ of both sides of the above equation as $y\to x'$, we conclude that $\overline{\lim}_{y\to x'}|f(y)-f(x')|\geq\frac{\epsilon}{2}>0$.\\ \\
 As $f$ is definable, there is a cylindrical decomposition $\Phi$ of $I^{\ell}$ where $|\Phi|=\comp$ such that $f$ is continuous on the cells of $\Phi$. Clearly, $\textnormal{Var}$ can only intersect the cells of $\Phi$ that are of dimension $\leq\ell-1$. Let $X\in\Phi$ such that $X\cap \textnormal{Var}\neq\emptyset$. We claim that $X$ necessarily has type $(0,\dots)$. Indeed, assume that  $\textnormal{type}(X)=(1,\dots)$. Then for any $t\in I$ the intersection $X\cap\{x_{1}=t\}$ has dimension $\leq\ell-2$. Let $t\in\pi_{1}(X)$. On the one hand, $\textnormal{Var}\cap \{x_{1}=t\}$ has dimension $\ell-1$ by the above. On the other hand, $\textnormal{Var}\cap \{x_{1}=t\}$ lies in the finite union of sets  $X_i\cap\{x_{1}=t\}$, where $X_i\in\Phi$ are the cells intersecting $\textnormal{Var}$ of type $(1,\dots)$  (as projections on $x_1$ of cells of $\Phi$ of type $(0,\dots)$ and cells of $\Phi$  of type $(1,\dots)$ do not intersect). Thus $\textnormal{Var}\cap \{x_1=t\}$ is of dimension $\leq\ell-2$, and this is a contradiction. 
 
 In conclusion, $\textnormal{Var}$ can only intersect cells of type $(0,\dots)$ and is therefore contained in $\comp$ hyperplanes $\{x_{1}=\textnormal{const}\}$, as needed.
\end{proof}
The following is a family version of Lemma \ref{lem:family_of_lipshitz}.
\begin{lem}
\label{lem:family_of_lipshitz_family_version}
Let $\{f_{\lambda}:I^{\ell}\to I\}_{\lambda\in I^{k}}$ be an $(\cc{F},D)$ family, and suppose that for every $x_{1}\in I,\;\lambda\in I^{k}$ the function $f_{\lambda}(x_{1},\cdot)$ is $L$-Lipshitz.Then there exists: \begin{itemize}
    \item A natural morphism $\varphi:\fort{K}^{k}\to I^{k}$ of format $\format$, degree $\comp$ where $|C(\fort{K}^{k})|=\comp$,
    \item for every $\cc{P}\in C(\fort{K}^{k})$ a fort $\fort{K}^{1}_{\cc{P}}$ of size $\comp$ and an $(\format,\comp)$-family $\{\phi'^{\lambda}:\fort{K}^{1}_{\cc{P}}\to I\}_{\lambda\in\varphi(\cc{P})}$ of natural morphisms,
\end{itemize}
such that if $\phi^{\lambda}$ is the extension of $\phi'^{\lambda}$ to $(\phi'^{\lambda})^{*}I^{\ell}$, then for every $\C\in C((\phi'^{\lambda})^{*}I^{\ell})$, the pullback $(\phi^{\lambda}|_{\C})^{*}f_{\lambda}$ is continuous.
\end{lem}
\begin{proof}
Let \begin{equation}
    \begin{split}
        d(\lambda,x_{1},x'_{1})&:=\sup_{y\in I^{\ell-1}}|f(x_{1},y)-f(x'_{1},y)|,\\
        \Sigma_{\lambda}&:=\{x_{1}:\overline{\lim}_{x'_{1}\to x_{1}}d(\lambda,x'_{1},x_{1})>0\}.
    \end{split}
\end{equation}
By Lemma \ref{lem:family_of_lipshitz}, for each fixed $\lambda\in I^{k}$, the function $f_{\lambda}$ is continuous outside the hyperplanes $\{x_{1}=c\}$ where $c\in\Sigma_{\lambda}$. We use Proposition \ref{prop:cyl_dec_forts} on the fort $I^{k+1}$ (with coordinates $(\lambda,x_{1})$) and the set $\Sigma:=\{(\lambda,x_{1}):x_{1}\in\Sigma_{\lambda}\}$ to obtain a natural morphism $\varphi':\fort{K}^{k+1}\to I^{k+1}$ compatible with $\Sigma$. \\ 

Define $\fort{K}^{k}:=\pi_{k}(\fort{K}^{k+1}),\;\varphi:=\varphi'_{1\dots k}$, and for every $\C\in C(\fort{K}^{k})$ let $\fort{K}^{1}_{\cc{P}}:=\fort{K}^{k+1}(\cc{P})$, and $\phi'^{\lambda}:=\varphi'_{\lambda}$. It is easy to see that this construction satisfies the requirements of the lemma.
\end{proof}
The following is a version of Lemma 12 from \cite{ALR} for sharp structures, which is a key lemma in this subsection.
\begin{lem}
\label{lem:boundedness_of_derivatives}
Let $f:I^{\ell}\to I$ be a definable $C^1$ function of format $\cc{F}$ and degree $D$, and suppose that for any $2\leq i\leq\ell$ one has $|\frac{\partial}{\partial x_i}f|\leq 1$. Then the function $\frac{\partial}{\partial x_{1}}f(x_1,\cdot)$ is bounded for all but $\comp$ values of $x_{1}$.
\end{lem}
\begin{proof}
  Consider the set \begin{equation}
 B:=\left\{x_{1}:\forall M>0\;\exists (x_{2},\dots,x_{\ell})\in I^{\ell-1}\;\left|\frac{\partial}{\partial x_{1}}f(x_{1},x_{2},\dots,x_{\ell})\right|>M\right\}
  \end{equation}
By the axioms of \so-minimal structures and Remark \ref{rem:format_degree_of_derivatives} the set $B$ has format $\format$ and degree $\comp$. By Lemma 12 from \cite{ALR}, the set $B$ is finite. Therefore it consists of $\comp$ points, as needed. 
\end{proof}
The following is a family version of Lemma \ref{lem:boundedness_of_derivatives}.
\begin{lem}
\label{lem:boundedness_of_derivatives_family_version}
Let $\{f_{\lambda}:I^{\ell}\}_{\lambda\in I^{k}}$ be an $(\cc{F},D)$-family of definable $C^{1}$ functions. Suppose that for every $2\leq i\leq s,\;\lambda\in I^{k}$ one has $|\frac{\partial}{\partial x_i}f|\leq 1$. Then there exists \begin{itemize}
    \item A natural morphism $\varphi:\fort{K}^{k}\to I^{k}$ of format $\format$, degree $\comp$ where $|C(\fort{K}^{k})|=\comp$,
    \item for every $\cc{P}\in C(\fort{K}^{k})$ a fort $\fort{K}^{1}_{\cc{P}}$ of size $\comp$ and an $(\format,\comp)$-family $\{\phi'^{\lambda}:\fort{K}^{1}_{\cc{P}}\to I\}_{\lambda\in\varphi(\cc{P})}$ of natural morphisms
\end{itemize}
such that if $\phi^{\lambda}$ is the extension of $\phi'^{\lambda}$ to $(\phi'^{\lambda})^{*}I^{\ell}$, then for every $\C\in C((\phi'^{\lambda})^{*}I^{\ell})$ and every fixed $x_{1}\in I$ the derivative  $\frac{\partial}{\partial x_{1}}\left((\phi^{\lambda}|_{\C})^{*}f_{\lambda}(x_{1},\cdot)\right)$ is bounded. 
\end{lem}
\begin{proof}
This lemma follows from Lemma \ref{lem:boundedness_of_derivatives} in the same way as Lemma \ref{lem:family_of_lipshitz_family_version} follows from Lemma \ref{lem:family_of_lipshitz}. 
\end{proof}
\subsection{Reduction to the case where \texorpdfstring{$f_{\C,j,\lambda}$}{Lg}  is \texorpdfstring{$C^{r}$}{Lg} and \texorpdfstring{$f_{\C,j,\lambda}(x_1,\cdot)$}{Lg} is an \texorpdfstring{$r$}{Lg}-function for all \texorpdfstring{$x_1$}{Lg}.} 
Let $\F^{\ell}$ be a fort and for every $\C\in C(\F)$ let $F_{\C,\lambda}=\{f_{\C,j,\lambda}:\C\to I\}_{\lambda\in I^{k},j\in J}$ be a collection of $|J|$ families such that for every fixed $j\in J$ the family $\{f_{\C,j,\lambda}:\C\to I\}_{\lambda\in I^{k}}$ is an $(\cc{F},D)$-family of definable functions. By the tower construction and the induction hypothesis we may assume that $\F^{\ell}=I\times\F^{\ell-1}$. Let $\C\in C(\F^{\ell-1})$ and define $F_{\C,(\lambda,x_{1})}=\{f_{I\times\C,j,\lambda}(x_{1},\cdot):\;f_{I\times\C,j,\lambda}\in F_{I\times\C,\lambda}\}$. That is, we consider each of the target families of functions $\{f_{I\times\C,j,\lambda}\}$ as a family of functions on  cells of $\F^{\ell-1}$ with parameters $\lambda,x_{1}$.\\

We use $F_{\ell-1,k+1}$ on $\F^{\ell-1}$ and $F_{\C,(\lambda,x_{1})}$. We obtain a natural morphism $\varphi':\fort{K}^{k+1}\to I^{k+1}$ (where the coordinates of $I^{k+1}$ are $(\lambda,x_{1})$) and for each $\cc{P}'\in C(\fort{K}^{k+1})$ a family of combinatorially equivalent $r$-morphisms $\phi^{(\lambda,x_{1})}:\fort{K}_{\cc{P}'}\to\F^{\ell-1}$ such that for every $(\lambda,x_{1})\in\varphi'(\cc{P}')$ and every $j\in J$ the function $(\phi^{(\lambda,x_{1})}|_{\C'})^{*}f_{I\times\C,j,\lambda}(x_1,\cdot)$ is an $r$-function whenever $\C'\in C(\fort{K}_{\cc{P}'})$ and $\phi^{(\lambda,x_{1})}(\C')\subset\C$.\\ 

Define $\fort{K}^{k}:=\pi_{k}(\fort{K}^{k+1})$, and fix a cell $\cc{P}\in C(\fort{K}^{k})$. We in particular have a combinatorially equivalent family $\{\varphi^{\lambda}:\fort{K}^{k}(\cc{P})\to I\}_{\lambda\in\varphi'_{1\dots k}(\cc{P})}$ given by $\varphi^{\lambda}:=\phi'_{\lambda}$. By Proposition \ref{prop:structure} we can identify $C(\fort{K}^{k}(\cc{P}))$  with the cells of $\fort{K}^{k+1}$ lying above $\cc{P}$. Define $s:C(\fort{K}^{k}(\cc{P}))\to \textnormal{Forts}(\ell-1)$ by $s(\cc{P}'):=\fort{K}_{\cc{P}'}$ and  $\fort{K}_{\cc{P}}:=\fort{K}^{k}(\cc{P})\ltimes s$.\\

Finally, define $\phi^{\lambda}:\fort{K}_{\cc{P}}\to\F^{\ell}$ by $\phi^{\lambda}(x_1,\dots,x_{\ell}):=(\varphi^{\lambda}(x_{1}),\phi^{(\lambda,x_{1})}(x_{2},\dots,x_{\ell})).$ We would like to use Proposition \ref{prop:family_of_combinatorially_equivalent_morphisms_is_morphism} to assert that the $\phi^{\lambda}$ are morphisms, for as $x_{1}$ ranges over a cell of $\fort{K}^{k}(\cc{P})$, the morphisms $\phi^{\lambda}(x_{1},\cdot)$ are combinatorially equivalent. However, we do not know that $\phi^{(\lambda,x_{1})}$ are continuous as functions of $x_{1},\dots,x_{\ell}$. Still, we do know that for every $x_{1}$ the morphism $\phi^{\lambda}(x_{1},\cdot)$ is an $r$-morphism, and in particular $1$-Lipshitz. Therefore, by Lemma \ref{lem:family_of_lipshitz_family_version}, we may assume that $\phi^{(\lambda,x_{1})}$ is continuous as a function of $x_1,\dots,x_{\ell}$. And so, by Proposition \ref{prop:family_of_combinatorially_equivalent_morphisms_is_morphism} the map $\phi^{\lambda}$ is a morphism for every $\lambda\in\varphi'_{1\dots k}(\cc{P})$ . We have thus reduced to the case where for every $\lambda\in I^{k},\;x_{1}\in\pi_{1}(\F^{\ell}),j\in J$ and $\C\in C(\F^{\ell})$ the function $f_{\C,j,\lambda}(x_1,\cdot)$ is an $r$-function.\\ \\ 
The last goal of this subsection is to further reduce to the case where the functions $f_{\C,j,\lambda}$ are $C^{r}$. Due to Proposition \ref{prop:smoothing} we can find a natural morphism $\varphi:\fort{K}^{k}\to I^{k}$ and for each cell $\cc{P}\in C(\fort{K}^{k})$ a family of combinatorially equivalent $C^{r}$ morphisms $\{\phi^{\lambda}:\fort{K}_{\cc{P}}\to\F^{\ell}\}_{\lambda\in\varphi(\cc{P})}$ such that $(\phi^{\lambda}|_{\C'})^{*}f_{\C,j,\lambda}$ is $C^{r}$ for every $j\in J,\lambda\in\varphi(\cc{P})$ whenever $\C'\in C(\fort{K}_{\cc{P}})$ and $\phi^{\lambda}(\C')\subset\C$. By $S_{\ell,k}$ we may assume that $\phi^{\lambda}$ are $r$-morphisms. Note that by the chain rule, since the first coordinate of a cellular map depends only on $x_{1}$, we see that $||(\phi^{\lambda}|_{\C'})^{*}f_{\C,j,\lambda}(x_1,\cdot)||_{r}=O_{r,\ell}(1)$ and so up to a linear subdivision of order $d=O_{\ell,r}(1)$ the reduction to $C^{r}$ functions has not spoiled our first reduction to $r$-morphisms for every $x_{1}$.
\subsection{Induction on the first unbounded derivative}

Order the indices $\alpha\in\bb{N}^{\ell}$ by lexicographic order. Let $\alpha$ be the first index such that $|\alpha|\leq r$ and such that there exists $\C\in C(\F^{\ell})$, $j\in J$ and $\lambda\in I^{k}$ such that $||f_{\C,j,\lambda}||_{r}>1$. By the previous subsection $\alpha_{1}>0$. By Lemma \ref{lem:boundedness_of_derivatives_family_version}, the tower construction and the induction hypothesis, we may assume that $\F^{\ell}=I\times\F^{\ell-1}$ and for every $x_{1}\in I$, and every $\C,j,\lambda$ the function $f_{\C,j,\lambda}^{(\alpha)}(x_{1},\cdot)$ is bounded. \\

We shall now reparametrize $x_{1}$. For every $\C\in C(\F^{\ell})$ define \begin{equation}
    S_{\C,j,\lambda}:=\{|f^{(\alpha)}_{\C,j,\lambda}(x_1,\dots,x_{\ell})|\geq\frac{1}{2}\cdot\underset{x_{2\dots\ell}}{\sup}|f^{(\alpha)}_{\C,j,\lambda}(x_1,\cdot)|\}\subset\C
\end{equation}
By Proposition \ref{prop:Effective_definable_choice} there exists definable families of curves $\{\gamma^{\C,j,\lambda}:I\to S_{\C,j,\lambda}\}_{\lambda\in I^{k}}$ such that $\gamma^{\C,j,\lambda}_{1}(x_{1})=x_{1}$. Let 
\begin{equation}
\begin{split}
    F^{1}_{\lambda}&:=\{\gamma^{\C,j,\lambda}:\C\in C(\F^{\ell}),j\in J,\lambda\in I^{k}\},\\
    F^{2}_{\lambda}&:=\{f^{(\alpha)}_{\C,j,\lambda}\circ\gamma^{\C,j,\lambda}:\C\in C(\F^{\ell}),j\in J,\lambda\in I^{k}\},\\
    F_{\lambda}&:=F^{1}_{\lambda}\cup F^{2}_{\lambda}.
\end{split}
\end{equation} 
We apply $F_{1,k}$ to the fort $I$ and the functions in $F_{\lambda}$. We obtain a natural morphism $\varphi:\fort{K}^{k}\to I^{k}$ and for every $\cc{P}\in C(\fort{K}^{k})$ a family of combinatorially equivalent morphisms $\phi'^{\lambda}:\fort{K}'_{\cc{P}}\to I$ such that for every $\C'\in C(\fort{K}'_{\cc{P}})$ the pullbacks $(\phi'^{\lambda}|_{\C'})^{*}\gamma^{\C,j,\lambda}$ and $(\phi'^{\lambda}|_{\C'})^{*}(f^{(\alpha)}_{\C,j,\lambda}\circ\gamma^{\C,j,\lambda})$ are $r$-functions. Define $\fort{K}_{\cc{P}}:=(\phi'^{\lambda})^{*}\F^{\ell}$ and let $\phi^{\lambda}:\fort{K}_{\cc{P}}\to\F^{\ell}$ extend $\phi'^{\lambda}$.\\ 

Let $\C''\in C(\fort{K}_{\cc{P}})$ such that $\phi^{\lambda}(\C'')\subset\C$.  By the induction on 
$\alpha$ and the chain rule we have that  $\left((\phi^{\lambda}|_{\C''})^{*}f_{\C,j,\lambda}\right)^{(\beta)}=O_{\ell,r}(1)$ when $\beta<\alpha$. When computing $\left((\phi^{\lambda}|_{\C''})^{*}f_{\C,j,\lambda}\right)^{(\alpha)}$ by the chain rule, again by the induction on $\alpha$ all the terms except for $(\phi^{\lambda}|_{\C''})^{\alpha_{1}}\cdot(\phi^{\lambda}|_{\C''})^{*}f_{\C,j,\lambda}^{(\alpha)}$ are bounded by $O_{\ell,r}(1)$. Now \begin{equation}
\begin{split}
  (\phi^{\lambda}|_{\C''})^{\alpha_{1}}\cdot(\phi^{\lambda}|_{\C''})^{*}f_{\C,j,\lambda}^{(\alpha)}&\leq  \left(\frac{\partial\phi^{\lambda}|_{\C''}}{\partial x_{1}}\right)^{\alpha_{1}}\cdot2(\phi_{1}^{\lambda}|_{\C''})^{*}(f_{\C,j,\lambda}^{(\alpha)}\circ\gamma^{\C,j,\lambda})\leq \\
  &\leq \frac{\partial\phi^{\lambda}|_{\C''}}{\partial x_{1}}\cdot2(\phi_{1}^{\lambda}|_{\C''})^{*}(f_{\C,j,\lambda}^{(\alpha)}\circ\gamma^{\C,j,\lambda}).
\end{split}
\end{equation}
We proceed to bound the right hand side. Consider $(\phi_{1}^{\lambda}|_{\C''})^{*}(f^{(\alpha-1_{1})}_{\C,j,\lambda}\circ\gamma^{\C,j,\lambda})'$, which is bounded by $O_{\ell,r}(1)$ by the induction hypothesis. By the chain rule it is equal to \begin{equation}
    \frac{\partial\phi^{\lambda}|_{\C''}}{\partial x_{1}}\cdot(\phi^{\lambda}_{1}|_{\C''})^{*}\left(f^{(\alpha)}_{\C,j,\lambda}\circ\gamma^{\C,j,\lambda}+\sum^{\ell}_{i=2}(\gamma^{\C,j,\lambda}_{j})'\cdot f^{(\alpha-1_{1}+1_{j})}_{\C,j,\lambda}\circ\gamma^{\C,j,\lambda}\right).
\end{equation}
Due to the induction on $\alpha$ and the definition of $\phi^{\lambda}$, all of the terms above except for $\frac{\partial\phi^{\lambda}|_{\C''}}{\partial x_{1}}\cdot(\phi_{1}^{\lambda}|_{\C''})^{*}(f_{\C,j,\lambda}^{(\alpha)}\circ\gamma^{\C,j,\lambda})$ are bounded by $O_{\ell,r}(1)$. Thus, $\frac{\partial\phi^{\lambda}|_{\C''}}{\partial x_{1}}\cdot2(\phi_{1}^{\lambda}|_{\C''})^{*}(f_{\C,j,\lambda}^{(\alpha)}\circ\gamma^{\C,j,\lambda})$ is also bounded by $O_{\ell,r}(1)$, and we are done by linear subdivision of order $d=O_{\ell,r}(1)$. \qedsymbol
\bibliographystyle{plain}
\bibliography{Ref}
\end{document}